\documentclass[11pt]{article}
\usepackage{amssymb,mathrsfs,verbatim,graphicx,rotate}
\usepackage{amsmath,amsthm}
\usepackage{graphics}
\usepackage{color}
\usepackage{float}
\usepackage{lscape}
\usepackage{setspace}
\usepackage{subfigure}
\usepackage{natbib}
\usepackage{multirow}
\usepackage{enumerate}
\usepackage[mathscr]{euscript}
\usepackage{prodint}

\makeatletter
\renewcommand{\@seccntformat}[1]{%
  \ifcsname format@#1\endcsname
    \csname format@#1\endcsname
  \else
    \csname the#1\endcsname\quad 
  \fi
}
\g@addto@macro\appendix{%
  \def\format@section{Appendix \thesection: }%
}
\makeatother

\numberwithin{equation}{section}
\newtheorem{theorem}{Theorem}
\newtheorem{proposition}{Proposition}
\newtheorem{lemma}{Lemma}
\newtheorem{corollary}{Corollary}

\DeclareMathOperator*{\argmax}{argmax}
\DeclareMathOperator*{\argmin}{argmin}

\newcommand{\bs}[1]{\boldsymbol{#1}}
\newcommand{\s}[1]{\mathscr{#1}}
\renewcommand{\d}[1]{\mathbb{#1}}

\DeclareMathOperator*{\inprob}{\stackrel{P}{\longrightarrow}}
\DeclareMathOperator*{\inproblow}{\rightarrow_{P}}

\DeclareMathOperator*{\indist}{\stackrel{d}{\longrightarrow}}
\newcommand{\bounded}{O_P}
\newcommand{\fasterthan}{o_P}
\newcommand{\fasterthandet}{o}
\newcommand{\boundeddet}{O}

\usepackage{array}
\newcolumntype{x}[1]{%
>{\centering\hspace{0pt}}p{#1}}%
\usepackage[bottom,flushmargin,hang,multiple]{footmisc}
\usepackage[margin=1in]{geometry}

\begin{document}

\title{Correcting an estimator of a multivariate monotone function\\ with isotonic regression}
\author{Ted Westling \\ Department of Mathematics and Statistics\\ University of Massachusetts Amherst \and Mark J. van der Laan \\ Department of Biostatistics\\ University of California, Berkeley \and Marco Carone\\ Department of Biostatistics\\ University of Washington }
\date{}

\maketitle

\begin{abstract}
In many problems, a sensible estimator of a possibly multivariate monotone function may itself fail to be monotone. We study the correction of such an estimator obtained via projection onto the space of functions monotone over a finite grid in the domain. We demonstrate that this corrected estimator has no worse supremal estimation error than the initial estimator, and that analogously corrected confidence bands contain the true function whenever the initial bands do, at no loss to average or maximal band width. Additionally, we demonstrate that the corrected estimator is uniformly asymptotically equivalent to the initial estimator provided that the initial estimator satisfies a stochastic equicontinuity condition and that the true function is Lipschitz and strictly monotone. We provide simple sufficient conditions for our stochastic equicontinuity condition in the important special case that the initial estimator is uniformly asymptotically linear, and illustrate the use of these results for estimation of a G-computed distribution function. Our stochastic equicontinuity condition is weaker than standard uniform stochastic equicontinuity, which has been required for alternative correction procedures. Crucially, this allows us to apply our results to the bivariate correction of the local linear estimator of a conditional distribution function known to be monotone in its conditioning argument.  Our experiments suggest that the projection step can yield significant practical improvements in performance for both the estimator and confidence band.%
\end{abstract}

%

\doublespacing

\section{Introduction}

\subsection{Background}

In many scientific problems, the parameter of interest is a component-wise monotone function. In practice, an estimator of this function may have several desirable statistical properties, yet fail to be monotone. This often occurs when the estimator is obtained through the pointwise application of a statistical procedure over the domain of the function. For instance, we may be interested in estimating a conditional cumulative distribution function $\theta_0$, defined pointwise as $\theta_0(a,y)=P_0(Y \leq y \mid A = a)$, over its domain $\mathscr{D}\subset \mathbb{R}^2$. Here, $Y$ may represent an outcome and $A$ an exposure. The map $y \mapsto \theta_0(a,y)$ is necessarily monotone for each fixed $a$. In some scientific contexts, it may be known that $a \mapsto \theta_0(a,y)$ is also monotone for each $y$, in which case $\theta_0$ is a bivariate component-wise monotone function. An estimator of $\theta_0$ can be constructed by estimating the regression function $(a,y) \mapsto E_{P_0} \left[ I(Y \leq y) \mid A = a\right]$ for each $(a,y)$ on a finite grid using kernel smoothing, and performing suitable interpolation elsewhere. For some types of kernel smoothing, including the Nadaraya-Watson estimator, the resulting estimator is necessarily monotone as a function of $y$ for each value of $a$, but not necessarily monotone as a function of $a$ for each value of $y$. For other types of kernel smoothing, including the local linear estimator, which often has smaller asymptotic bias than the Nadaraya-Watson estimator, the resulting estimator need not be monotone in either component.

Whenever the function of interest is component-wise monotone, failure of an estimator to itself be monotone can be problematic. This is most apparent if the monotonicity constraint is probabilistic in nature -- that is, the parameter mapping is monotone under all possible probability distributions. This is the case, for instance, if $\theta_0$ is a distribution function. In such settings, returning a function estimate that fails to be monotone is nonsensical, like reporting a probability estimate outside the interval $[0,1]$. However, even if the monotonicity constraint is based on scientific knowledge rather than probabilistic constraints, failure of an estimator to be monotone can be an issue. For example, if the parameter of interest represents average height or weight  among children as a function of age, scientific collaborators would likely be unsatisfied if presented with an estimated curve that were not monotone.  Finally, as we will see, there are often finite-sample performance benefits to ensuring that the monotonicity constraint is respected.

Whenever this phenomenon occurs, it is natural to seek an estimator that respects the monotonicity constraint but nevertheless remains close to the initial estimator, which may otherwise have good statistical properties. A monotone  estimator  can be naturally constructed by projecting the initial estimator onto the space of monotone functions with respect to some norm. A common choice is the $L_2$-norm, which amounts to using multivariate isotonic regression to correct the initial estimator. 

\subsection{Contribution and organization of the article}

In this article, we discuss correcting an initial estimator of a multivariate monotone function by computing the isotonic regression of the estimator over a finite grid in the domain, and  interpolating between grid points. We also consider correcting an initial confidence band by using the same procedure applied to the upper and lower limits of the band.  We provide three general results regarding this simple procedure.
\begin{enumerate}
\item Building on the results of \cite{robertson1988order} and \cite{chernozhukov2009rearrangement}, we demonstrate that the corrected estimator is at least as good as the initial estimator, meaning:
\begin{enumerate}[(a)]
\item its uniform error over the grid used in defining the projection is less than or equal to that of the initial estimator for every sample; 
\item its uniform error over the entire domain is less than or equal to that of the initial estimator asymptotically;
\item the corrected confidence band contains the true function on the projection grid whenever the initial band does, at no cost in terms of average or uniform band width.
\end{enumerate}

\item We provide high-level sufficient conditions under which the uniform difference between the initial and corrected estimators is $\fasterthan(r_n^{-1})$ for a generic sequence $r_n \longrightarrow \infty$.

\item We provide simpler lower-level sufficient conditions in two special cases:
\begin{enumerate}[(a)]
 \item when the initial estimator is uniformly asymptotically linear, in which case the appropriate rate is $r_n = n^{1/2}$;
 \item when the initial estimator is kernel-smoothed with bandwidth $h_n$, in which case the appropriate rate is $r_n = (n h_n)^{1/2}$ for univariate kernel smoothing.
 \end{enumerate}
\end{enumerate}

We apply our theoretical results to two sets of examples: nonparametric efficient estimation of a G-computed distribution function for a binary exposure, and local linear estimation of a conditional distribution function with a continuous exposure.

Other authors have considered the correction of an initial estimator using isotonic regression. To name a few, \cite{mukarjee1994} used a projection-like procedure applied to a kernel smoothing estimator of a regression function, whereas \cite{patra2016} used the projection procedure applied to a univariate cumulative distribution function in the context of a mixture model. These articles addressed the properties of the projection procedure in their specific applications. In contrast, we provide general results that are applicable broadly.

\subsection{Alternative projection procedures}

The projection approach is not the only possible correction procedure. \cite{dette2006}, \cite{chernozhukov2009rearrangement}, and \cite{chernozhukov2010quantile} studied a correction based on monotone rearrangements. However, monotone rearrangements do not generalize to the multivariate setting as naturally as projections -- for example, \cite{chernozhukov2009rearrangement} proposed averaging a variety of possible multivariate monotone rearrangements to obtain a final monotone estimator. In contrast, the $L_2$ projection of an initial estimator onto the space of monotone functions is uniquely defined, even in the context of multivariate functions.

\cite{daouia2012isotonic} proposed an alternative correction procedure that consists of taking a convex combination of upper and lower monotone envelope functions, and they demonstrated conditions under which their estimator is asymptotically equivalent in supremum norm to the initial estimator. There are several differences between our contributions and those of \cite{daouia2012isotonic}. For instance, \cite{daouia2012isotonic} did not study correction of confidence bands, which we consider in Section~\ref{confidence}, or the important special case of asymptotically linear estimators, which we consider in Section~\ref{linear}. Our results in these two sections apply equally well to our correction procedure and to the correction procedure considered by \cite{daouia2012isotonic}.

Perhaps the most important theoretical contribution of our work  beyond that of existing research is the weaker form of stochastic equicontinuity that we require for establishing asymptotic equivalence of the initial and projected estimators. In contrast,  \cite{daouia2012isotonic} explicitly required the usual uniform asymptotic equicontinuity, while application of the Hadamard differentiability results of  \cite{chernozhukov2010quantile} requires weak convergence to a tight limit, which is stronger than uniform asymptotic equicontinuity.  Our weaker condition allows us to use our general results to tackle a broader range of initial estimators, including kernel smoothed estimators, which are typically not uniformly asymptotically equicontinuous at useful rates, but nevertheless can frequently be shown to satisfy our condition. We discuss this in detail in Section~\ref{sec:kern}. We illustrate this general contribution in Section~\ref{cond_dist} by studying the bivariate correction of a conditional distribution function estimated using local linear regression, which would not be possible using the stronger asymptotic equicontinuity condition. In numerical studies, we find that the projected estimator and confidence bands can offer substantial finite-sample improvements over the initial estimator and bands in this example.

%

\section{Main results}\label{monotone}
\subsection{Definitions and statistical setup}

Let $\s{M}$ be a statistical model of probability measures on a probability space $(\s{X}, \s{B})$. Let $\theta : \s{M} \to  \ell^{\infty}(\s{T})$ be a parameter of interest on $\mathscr{M}$, where $\s{T}:= [0,1]^d$ and $\ell^{\infty}(\s{T})$ is the Banach space of bounded functions from $\s{T}$ to $\d{R}$ equipped with supremum norm $\|\cdot\|_{\s{T}}$. We have specified this particular $\s{T}$ for simplicity, but the results established here apply to any bounded rectangular domain $\s{T}\subset \d{R}^d$. For each $P \in \s{M}$, denote by $\theta_P$ the evaluation of $\theta$ at $P$ and note that $\theta_P$ is a bounded real-valued function on $\s{T}$.  For any $t \in \s{T}$, denote by $\theta_P(t) \in \d{R}$ the evaluation of $\theta_P$ at $t$.

For any vector $t\in\d{R}^d$ and $1\leq j\leq d$, denote by $t_j$ the $j^{th}$ component of $t$. Define the partial order $\leq$ on $\d{R}^d$ by setting $t \leq t'$ if and only if $t_j \leq t_j'$ for each $1 \leq j \leq d$.  A function $f: \d{R}^d \to \d{R}$ is called (component-wise) monotone non-decreasing if $t \leq t'$ implies that $f(t) \leq f(t')$. Denote $\|t\| = \max_{1 \leq j \leq d} |t_j|$ for any  vector $t \in \d{R}^d$. Additionally, denote by $\bs\Theta \subset \ell^{\infty}(\s{T})$ the convex set of bounded monotone non-decreasing functions from $\s{T}$ to $\d{R}$. For concreteness, we focus on non-decreasing functions, but all results established here apply equally to non-increasing functions.

Let $\mathscr{M}_0:=\{P\in\mathscr{M}:\theta_P\in\bs{\Theta}\}\subseteq \s{M}$ and suppose that $\mathscr{M}_0$ is nonempty. Generally, this inclusion is strict only if, rather than being implied by the rules of probability, the monotonicity constraint stems at least in part from prior scientific knowledge. Also, define $\bs\Theta_0 := \{ \theta \in \bs\Theta : \theta = \theta_P\mbox{ for some }P \in \mathscr{M}\} \subseteq \bs\Theta$. We are primarily interested in settings where $\bs\Theta_0 = \bs\Theta$, since in this case there is no additional knowledge about $\theta$ encoded by $\s{M}$, and in particular there is no danger of yielding a corrected estimator that is compatible with no $P \in \s{M}$.

Suppose that observations $X_1,X_2,\ldots,X_n$ are sampled independently from an unknown distribution $P_0\in\s{M}_0$, and that we wish to estimate $\theta_0:=\theta_{P_0}$ based on these observations. Suppose that, for each $t \in \s{T}$, we have access to an estimator $\theta_n(t)$ of $\theta_0(t)$ based on $X_1,X_2,\ldots,X_n$. We note that the assumption that the data are independent and identically distributed is not necessary for Theorems~\ref{thm:barlow} and~\ref{monotone_supnorm} below. For any suitable $f:\s{X}\rightarrow \mathbb{R}$, we define $Pf := \int f(x) \, P(dx)$ and $\mathbb{G}_nf := n^{1/2}\int f(x)(\d{P}_n-P_0)(dx)$, where $\d{P}_n$ is the empirical distribution based on $X_1, X_2,\dotsc, X_n$.

The central premise of this article is that $\theta_n(t)$ may have desirable statistical properties for each $t$ or even uniformly in $t$, but that $\theta_n$ as an element of $\ell^{\infty}(\s{T})$ may not fall in $\bs\Theta$ for any finite $n$ or even with probability tending to one. Our goal is to provide a corrected estimator $\theta_n^*$ that necessarily falls in $\bs\Theta$, and yet retains the statistical properties of $\theta_n$. A natural way to accomplish this is to define $\theta_n^*$ as the closest element of $\bs\Theta$ to $\theta_n$ in some norm on $\s{T}$. Ideally, we would prefer to take $\theta_n^*$ to minimize $\| \theta-\theta_n\|_{\s{T}}$ over $\theta \in \bs\Theta$. However, this is not tractable for two reasons. First, optimization over the entirety of $\s{T}$ is an infinite-dimensional optimization problem, and is hence frequently computationally intractable. To resolve this issue, for each $n$, we let $\s{T}_n = \{t_1,t_2, \dotsc, t_{m_n}\} \subseteq \s{T}$ be a finite rectangular lattice in $\s{T}$ over which we will perform the optimization, and define and consider $\|\cdot\|_{\s{T}_n}$ as the supremum norm over $\s{T}_n$.  While it is now computationally feasible to define $\theta_{n,\infty}^*$ as a minimizer over $\theta \in \bs\Theta$ of the finite-dimensional objective function $\| \theta - \theta_n\|_{\s{T}_n}$, this objective function is challenging due to its non-differentiability. Instead, we define
\begin{equation} \theta_n^* \in \argmin_{\theta \in \bs\Theta} \sum_{t \in \s{T}_n} \left[\theta(t) -\theta_n(t)\right]^2\ .\label{projection}\end{equation}
The squared-error objective function is smooth in its arguments. In dimension $d=1$, $\theta_n^*$ thus defined is simply the isotonic regression of $\theta_n$ on the grid $\s{T}_n$, which has a closed-form representation as the greatest convex minorant of the so-called cumulative sum diagram. Furthermore, since $\| \theta_n^* - \theta_n\|_{\s{T}_n} \geq \| \theta_{n,\infty}^* - \theta_n\|_{\s{T}_n}$, many of our results also apply to $\theta_{n,\infty}^*$.

We note that $\theta_n^*$ is only uniquely defined on $\s{T}_n$. To completely characterize $\theta_n^*$, we must monotonically interpolate function values between elements of $\s{T}_n$. We will permit any monotonic interpolation that satisfies a weak condition. By the definition of a rectangular lattice, every $t \in \s{T}$ can be assigned a hyper-rectangle whose vertices $\{s_1, s_2\dotsc, s_{2^d}\}$ are elements of $\s{T}_n$ and whose interior has empty intersection with $\s{T}_n$. If multiple such hyper-rectangles exist for $t$, such as when $t$ lies on the boundary of two or more such hyper-rectangles, one can be assigned arbitrarily. We will assume that, for $t \notin \s{T}_n$, $\theta_n^*(t) = \sum_k \lambda_{k,n}(t) \theta_n^*(s_k)$ for weights $ \lambda_{1,n}(t),\lambda_{2,n}(t),\ldots,\lambda_{2^d,n}(t)\in(0,1)$ such that $\sum_k \lambda_{k,n}(t) = 1$. In words, we assume that $\theta_n^*(t)$ is a convex combination of the values of $\theta_n^*$ on the vertices of the hyper-rectangle containing $t$. A simple interpolation approach consists of setting $\theta_n^*(t) = \theta_n^*(t')$ with $t'$ the element of $\s{T}_n$ closest  to $t$, and choosing any such element if there are multiple elements of $\s{T}_n$ equally close to $t$. This particular scheme satisfies our requirement.

Finally, for each $n$, we let $\ell_n(t) \leq u_n(t)$ denote lower and upper endpoints of a confidence band for $\theta_0(t)$. We  then define $\ell_n^*$ and $u_n^*$ as the corrected versions of $\ell_n$ and $u_n$ using the same projection and interpolation procedure defined above for obtaining $\theta_n^*$ from $\theta_n$. 

In dimension $d = 1$, $\theta_n^*(t)$, $\ell_n^*(t)$, and $u_n^*(t)$  can be obtained for $t \in \s{T}_n$ via the Pool Adjacent Violators Algorithm, as implemented in the \texttt{R} command \texttt{isoreg} \citep{Rlang}. In dimension $d = 2$, the corrections can be obtained using the algorithm described in \cite{bril1984bivariate}, which is implemented in the \texttt{R} command \texttt{biviso} in the package \texttt{Iso} \citep{Isopackage}. In dimensions $d \geq 3$, no tailored algorithm for computation of the isotonic regression estimate yet exists to our knowledge. However, general-purpose algorithms for minimization of quadratic criteria over convex cones have been developed an implemented in the \texttt{R} package \texttt{coneproj} and may be used in this case \citep{meyer1999cone, coneproj}.

\subsection{Properties of the projected estimator}

The projected estimator $\theta_n^*$ is the isotonic regression of $\theta_n$ over the grid $\s{T}_n$. Hence, many existing finite-sample results on isotonic regression can be used to deduce properties of $\theta_n^*$. Theorem~\ref{thm:barlow} below collects a few of these properties, building upon the results of \cite{barlow1972order} and \cite{chernozhukov2009rearrangement}. We denote $\omega_n := \sup_{t \in \s{T}} \min_{s\in\s{T}_n} \| t - s\|$ as the mesh of $\s{T}_n$ in $\s{T}$.
\begin{theorem}\label{thm:barlow}
\begin{itemize}
\item[(i)] It holds that $\|\theta_n^* - \theta_0\|_{\s{T}_n} \leq \|\theta_n - \theta_0\|_{\s{T}_n}$.
\item[(ii)] If $\omega_n = \fasterthan(1)$ and $\theta_0$ is continuous on $\s{T}$, then $\| \theta_n^* - \theta_0\|_{\s{T}} \leq \| \theta_n - \theta_0\|_{\s{T}} + \fasterthan(1)$. 
\item[(iii)] If there exists some $\alpha>0$ for which $\sup_{s,t\in\s{T}:\|t - s\| \leq \delta} |\theta_0(t) - \theta_0(s)| = \fasterthandet(\delta^\alpha)$ as $\delta \to 0$, then $\| \theta_n^* - \theta_0\|_{\s{T}} \leq \|\theta_n - \theta_0\|_{\s{T}} + \fasterthan(\omega_n^{\alpha})$.
\item[(iv)] Whenever $\theta_0(t)\in[\ell_n(t),u_n(t)]$ for all $t\in\s{T}_n$, $\theta_0(t) \in [\ell_n^*(t), u_n^*(t)]$ for all $t\in\s{T}_n$.
\item[(v)] It holds that $\sum_{t \in \s{T}_n} [ u_n^*(t) - \ell_n^*(t)] = \sum_{t \in \s{T}_n} [ u_n(t) - \ell_n(t)]$ and $\| u_n^* - \ell_n^* \|_{\s{T}_n} \leq \|u_n - \ell_n \|_{\s{T}_n}$.
\end{itemize}
\end{theorem}
Before presenting the proof of Theorem~\ref{thm:barlow}, we remark briefly on its implications. Part (i) says that the estimation error of $\theta_n^*$ over the grid $\s{T}_n$ is never worse than that of $\theta_n$, whereas parts (ii) and (iii) say that the estimation error of $\theta_n^*$ on all of $\s{T}$ is asymptotically no worse than the estimation error of $\theta_n$ in supremum norm. Similarly, part (iv) says that the isotonized band $[\ell_n^*, u_n^*]$ never has worse coverage than the original band over $\s{T}_n$. Finally, part (v) says that the potential increase in coverage comes at no cost to the average or supremum width of the bands over $\s{T}_n$. We note that parts (i), (iv) and (v)  hold true for each $n$. 

While comprehensive in scope, Theorem~\ref{thm:barlow} does not rule out the possibility that $\theta_n^*$ performs strictly better, even asymptotically, than $\theta_n$, or that the band $[\ell_n^*, u_n^*]$ is asymptotically strictly more conservative than $[\ell_n, u_n]$. In order to  construct confidence intervals or bands with correct asymptotic coverage, a stronger result is needed: it must be that $\|\theta_n^* - \theta_n\|_{\s{T}} = \fasterthan(r_n^{-1})$, where $r_n$ is a diverging sequence such that $r_n \|\theta_n - \theta_0\|_{\s{T}}$ converges in distribution to a non-degenerate limit distribution. Then, we would have that $r_n\|\theta_n^* - \theta_0\|_{\s{T}}$ converges in distribution to this same limit, and hence confidence bands constructed using approximations of this limit distribution would have correct coverage when centered around $\theta_n^*$, as we discuss more below.

We consider the following conditions on $\theta_0$ and the initial estimator $\theta_n$:
\begin{description}
\item[(A)] there exists a deterministic sequence $r_n$ tending to infinity such that, for all $\delta>0$,
\[\sup_{\|t -s\| < \delta/r_n} \left|r_n\left[\theta_n(t) - \theta_0(t)\right] - r_n\left[\theta_n(s) - \theta_0(s)\right] \right| \inproblow 0;\]
\item[(B)] there exists $K_1 < \infty$ such that $|\theta_0(t) - \theta_0(s)| \leq K_1\|t-s\|$ for all $t,s\in \s{T}$;
\item[(C)] there exists $K_0 > 0$ such that $K_0\|t-s\| \leq |\theta_0(t) - \theta_0(s)|$ for all $t,s\in \s{T}$.
\end{description}
Based on these conditions, we have the following result.
\begin{theorem}\label{monotone_supnorm}
If conditions (A)--(C) hold and $\omega_n = \fasterthan(r_n^{-1})$, then $\|\theta_n^* - \theta_n\|_{\s{T}} = \fasterthan(r_n^{-1})$.
\end{theorem}
This result indicates that the projected estimator is uniformly asymptotically equivalent to the original estimator in supremum norm at the rate $r_n$.

Condition (A) is related to, but notably weaker than, uniform stochastic equicontinuity \citep[p.\ 37]{van1996weak}. (A) follows if, in particular, the process $\{r_n [\theta_n(t) - \theta_0(t)] : t\in \s{T}\}$ converges weakly to a tight limit in the space $\ell^{\infty}(\s{T})$. However, the latter condition is sufficient but not necessary for (A) to hold. This is important for application of our results to kernel smoothing estimators, which typically do not converge weakly to a tight limit, but for which condition (A) nevertheless often holds. We discuss this at length in Section~\ref{cond_dist}. The results of \cite{daouia2012isotonic} (see in particular condition (C3) therein) and \cite{chernozhukov2010quantile} rely on uniform stochastic equicontinuity in demonstrating asymptotic equivalence of their correction procedures, which essentially limits the applicability of their procedures to estimators that converge weakly to a tight limit in $\ell^{\infty}(\s{T})$.

Condition (B) constrains $\theta_0$ to be Lipschitz. Condition (C) constrains the variation of $\theta_0$ from below,  and is slightly more restrictive than a requirement for strict monotonicity. If, for instance, $\theta_0$ is differentiable, then (C) is satisfied if all first-order partial derivatives of $\theta_0$ are bounded away from zero. Condition (C) excludes, for instance, situations in which $\theta_0$ is differentiable with null derivative over an interval. In such cases, $\theta_n^*$ may have strictly smaller variance on these intervals than $\theta_n$ because $\theta_n^*$ will pool estimates across the flat region while $\theta_n$ may not. Hence, in such cases, $\theta_n^*$ may potentially asymptotically improve on $\theta_n$, so that $\theta_n^*$ and $\theta_n$ are not asymptotically equivalent at the rate $r_n$. Theoretical results in these cases would be of interest, but are beyond the scope of this article. In addition to conditions (A)--(C), Theorem~\ref{monotone_supnorm} requires that the mesh $\omega_n$ of $\s{T}_n$ tend to zero in probability faster than $r_n^{-1}$. Since $\s{T}_n$ is chosen by the user, this is not a problem in practice.

We prove Theorem~\ref{monotone_supnorm} via three lemmas, which may be of interest in their own right. The first lemma controls the size of deviations in $\theta_n$ over small neighborhoods, and does not hinge on condition (C) holding.
\begin{lemma}
If (A)--(B) hold and $b_n = \fasterthan(r_n^{-1})$, then $\displaystyle \sup_{\|t-s\| \leq b_n}   \left|  \theta_n(t) - \theta_n(s)\right| = \fasterthan(r_n^{-1})$.
\label{lemma:moduli_of_continuity}
\end{lemma}

The second lemma controls the size of neighborhoods over which violations in monotonicity can occur. Henceforth, we define $\kappa_n := \sup\left\{\|t-s\| : s, t \in \s{T}, s \leq t, \theta_n(t) \leq \theta_n(s)\right\}.$ In this lemma we again require condition (A), but now require (C) rather than (B).
\begin{lemma}\label{neighborhoods}
If conditions (A) and (C) hold, then $\kappa_n = \fasterthan(r_n^{-1})$.
\end{lemma}

Our final lemma bounds the maximal absolute deviation between $\theta_n^*$ and $\theta_n$ over the grid $\s{T}_n$ in terms of the supremal deviations of $\theta_n$ over neighborhoods smaller than $\kappa_n$. This lemma does not depend on any of the conditions (A)--(C).
\begin{lemma}
The inequality $\max_{t \in \s{T}_n} |\theta_n^*(t) - \theta_n(t)| \leq \sup_{\|s- t\| \leq \kappa_n} |\theta_n(s) - \theta_n(t)|$ holds.
\label{pava}
\end{lemma}

The proof of Theorem~\ref{monotone_supnorm} follows easily from Lemmas~\ref{lemma:moduli_of_continuity},~\ref{neighborhoods}, and~\ref{pava}. The proof of these Lemmas dn Theorem~\ref{monotone_supnorm} are presented in Appendix~B. 

\subsection{Construction of confidence bands}\label{confidence}

Suppose there exists a fixed function $\gamma_{\alpha} : \s{T} \to \d{R}$ such that $\ell_n$ and $u_n$ satisfy:
\begin{description}
\item[(a)] $\| r_n (\theta_n - \ell_n) - \gamma_\alpha \|_{\s{T}} \inproblow 0$,
\item[(b)] $\| r_n (u_n - \theta_n) - \gamma_\alpha \|_{\s{T}} \inproblow 0$,
\item[(c)] $P_0\left( r_n | \theta_n(t) - \theta_0(t)| \geq \gamma_\alpha(t)\mbox{ for all }t \in \s{T}\right)\longrightarrow 1 - \alpha$.
\end{description} As an example of a confidence band that satisfies conditions (a)--(c), suppose that $\sigma_0:\s{T}\rightarrow (0,+\infty)$ is a scaling function and $c_{\alpha}$ is a fixed constant such that, as $n$ tends to infinity, 
\[P_0\left(r_n \left\| \frac{\theta_n  - \theta_0}{\sigma_0} \right\|_{\s{T}} \geq c_{\alpha}\right)\longrightarrow 1 - \alpha\ .\] 
If $\sigma_n$ is an estimator of $\sigma_0$ satisfying $\| \sigma_n -\sigma_0\|_{\s{T}} \inproblow 0$ and $c_{\alpha,n}$ is an estimator of $c_{\alpha}$ such that $c_{\alpha,n} \inproblow c_{\alpha}$, then the Wald-type band defined by lower and upper endpoints $\ell_n(t) := \theta_n(t) - c_{\alpha,n}r^{-1}_n\sigma_n(t)$ and $u_n(t) := \theta_n(t) + c_{\alpha}r^{-1}_n\sigma_n(t)$ satisfies (a)--(c) with $\gamma_\alpha = c_{\alpha} \sigma_0$. However, the latter conditions can also be satisfied by other types of bands, such as those constructed with a consistent bootstrap procedure.

Under conditions (a)--(c), the confidence band $[\ell_n,u_n]$ has asymptotic coverage $1-\alpha$. When conditions (A) and (B) also hold, the corrected band $[\ell_n^*, u_n^*]$ has the same asymptotic coverage as the original band $[\ell_n, u_n]$, as stated in the following result.
\begin{corollary}\label{cor:band}
If conditions (A)--(B) and (a)--(c) hold, $\gamma_\alpha$ is uniformly continuous on $\s{T}$, and $\omega_n = \fasterthan(r_n^{-1})$, then the confidence band $[\ell_n^*, u_n^*]$ has asymptotic coverage $1-\alpha$.
\end{corollary}
The proof of Corollary~\ref{cor:band} is presented in Appendix~\ref{app:cor1}. We also note that Theorem~\ref{monotone_supnorm} immediately implies that Wald-type confidence bands constructed around $\theta_n$ have the same asymptotic coverage if they are constructed around $\theta_n^*$ instead.

\section{Refined results under additional structure}

In this section, we provide more detailed conditions that imply condition (A) in two special cases: when $\theta_n$ is asymptotically linear, and when $\theta_n$ is a kernel smoothing-type estimator.

\subsection{Special case I: asymptotically linear estimators}\label{linear}

Suppose that the initial estimator $\theta_n$ is uniformly asymptotically linear: for each $t\in \s{T}$, there exists $\phi_{0,t}:\s{X}\mapsto \mathbb{R}$ depending on $P_0$ such that $\int \phi_{0,t}(x)dP_0(x)=0$, $\int \phi^2_{0,t}(x)dP_0(x)<\infty$, and 
\begin{equation}\label{asy_linear}
\theta_n(t) = \theta_0(t)+ \frac{1}{n}\sum_{i=1}^{n} \phi_{0,t}(X_i) + R_{n, t}
\end{equation}
for a remainder term $R_{n,t}$ with $n^{1/2}\sup_{t\in \s{T}} |R_{n,t}| = \fasterthan(1)$. The function $\phi_{0,t}$ is the influence function of $\theta_n(t)$ under sampling from $P_0$. It is desirable for $\theta_n$ to have representation \eqref{asy_linear} because this immediately implies its uniform weak consistency  as well as the pointwise asymptotic normality of $n^{1/2}\left[\theta_n(t) - \theta_0(t)\right]$ for each $t \in \s{T}$.  If in addition the collection $\{\phi_{0,t} : t \in \s{T}\}$ of influence functions forms a $P_0$-Donsker class, $\{n^{1/2}\left[\theta_n(t) - \theta_0(t)\right] : t \in \s{T}\}$ converges weakly in $\ell^{\infty}(\s{T})$ to a Gaussian process with covariance function $\Sigma_0:(t,s)\mapsto\int \phi_{0,t}(x)\phi_{0,s}(x)dP_0(x)$. Uniform asymptotic confidence bands based on $\theta_n$ can then be formed by using appropriate quantiles from any suitable approximation of the distribution of the supremum of the limiting Gaussian process.

We introduce two additional conditions:
\begin{description}
\item[(A1)] the collection $\{\phi_{0,t} : t \in \s{T}\}$ of influence curves is a $P_0$-Donsker class;
\item[(A2)] $\Sigma_0$ is uniformly continuous in the sense that $\limsup_{\|t - s\| \to 0} |\Sigma_0(s, t) - \Sigma_0(t,t)| = 0.$
\end{description} Whenever $\theta_n$ is uniformly asymptotically linear, Theorem~\ref{monotone_supnorm} can be shown to hold under (A1), (A2) and (B), as implied by the theorem below. The validity of (A1) and (A2) can be assessed by scrutinizing the influence function $\phi_{0,t}$ of $\theta_n(t)$ for each $t \in \s{T}$. This fact renders the verification of these conditions very simple once uniform asymptotic linearity has been established.
\begin{theorem} For any estimator $\theta_n$ satisfying \eqref{asy_linear}, (A1) and (A2) together imply (A).
\label{thm:emp_process}
\end{theorem} 
The proof of Theorem~\ref{thm:emp_process} is provided in Appendix~\ref{app:thm3}. In Section~\ref{survival}, we illustrate the use of Theorem~\ref{thm:emp_process} for the estimation of a G-computed distribution function. 

We note that conditions (A1) and (A2) are actually sufficient to establish uniform asymptotic equicontinuity, which as discussed above is stronger than (A). Therefore, Theorem~\ref{thm:emp_process} can also be used to prove asymptotic equivalence of the majorization/minorization correction procedure studied in \cite{daouia2012isotonic}.

\subsection{Special case II: kernel smoothed estimators}\label{sec:kern}

For certain parameters,  asymptotically linear estimators are not available. In particular, this is the case when the parameter of interest is not sufficiently smooth as a mapping of $P_0$. For example, density functions, regression functions, and conditional quantile functions do not permit asymptotically linear estimators in a nonparametric model when the exposure is continuous. In these settings, a common approach to nonparametric estimation is kernel smoothing.

Recent results suggest that, as a process, the only possible weak limit of $\{ r_n[\theta_n(t) - \theta_0(t)] : t \in \s{T}\}$ in $\ell^{\infty}(\s{T})$ may be zero when $\theta_n$ is a kernel smoothed estimator. For example, in the case of the Parzen-Rosenblatt estimator of a density function with bandwidth $h_n$, Theorem 3 of \cite{stupfler2016} implies that if $c_n := r_n \left(n h_n / |\log h_n|\right)^{-1/2} \to 0$, then $\{ r_n[\theta_n(t) - \theta_0(t)] : t \in \s{T}\}$ converges weakly to  zero in $\ell^{\infty}(\s{T})$, whereas if $c_n \to c \in (0, \infty]$, then it does not converge weakly to a tight limit in $\ell^{\infty}(\s{T})$. As a result, $\{ r_n[\theta_n(t) - \theta_0(t)] : t \in \s{T}\}$ only satisfies uniform stochastic equicontinuity for $r_n$ such that $r_n \left(n h_n / |\log h_n|\right)^{-1/2} \to 0$. However, for any such rate $r_n$, $r_n^{-1}$ is slower than the pointwise and uniform rates of convergence of $\theta_n - \theta_0$. As a result, $\theta_n$ and $\theta_n^*$ may not be asymptotically equivalent at the uniform rate of convergence of $\theta_n - \theta_0$, so that confidence intervals and regions based on the limit distribution of $\theta_n - \theta_0$, but centered around $\theta_n^*$, may not have correct coverage. We note that, while \cite{stupfler2016} establishes formal results for the Parzen-Rosenblatt estimator, we expect that the results therein extend to a variety of kernel smoothed estimators.

As a result of the lack of uniform stochastic equicontinuity of $r_n(\theta_n - \theta_0)$ for useful rates $r_n$, establishing (A) is much more difficult for kernel smoothed estimators than for asymptotically linear estimators. However, since (A) is weaker than uniform stochastic equicontinuity, it may still be possible. Here, we provide alternative sufficient conditions that imply condition (A) and that we have found useful for studying a kernel smoothed estimator $\theta_n$.

When the initial estimator $\theta_n$ is kernel smoothed, we can frequently show that
\begin{equation}
\sup_{t \in \s{T}} \left| r_n \left[ \theta_n(t) - \theta_0(t)\right] - a_n b_0(t) -  R_n(t) \right| \inprob 0\ , \label{eq:smooth}
\end{equation}
 where $b_0 : \s{T} \to \d{R}$ is a deterministic bias, $a_n$ is sequences of positive constants, and $R_n :\s{T} \to \d{R}$ is a random remainder term. We then have
\begin{align*}
& \sup_{\| t - s\| < \delta / r_n} \left| r_n \left[ \theta_n(t) - \theta_0(t) \right] - r_n \left[ \theta_n(s) - \theta_0(s) \right] \right|  \\
&\qquad= \sup_{\| t - s\| < \delta / r_n} a_n \left| b_0(t) - b_0(s) \right| +\sup_{\| t - s\| < \delta / r_n}\left| R_n(t) - R_n(s) \right| +\fasterthan(1) \ .
\end{align*}
 If $b_0$ is uniformly continuous on $\s{T}$ and $a_n = O(1)$, or $b_0$ is uniformly $\alpha$-H{\"o}lder on $\s{T}$ and $a_n = O\left( r_n^\alpha\right)$, then the first term on the right hand side tends to zero in probability. Attention may then be turned to demonstrating that the second term vanishes in probability. It appears difficult to provide a general characterization of the form of $R_n$ that encompasses kernel smoothed estimators. However, in our experience, it is frequently the case that $R_n(t)$ involves terms of the form $\d{G}_n \nu_{n,t}$, where $\nu_{n,t} : \s{X} \to \d{R}$ is a deterministic function for each $n \in \{1, 2, \dots\}$ and $t \in \s{T}$. In the course of demonstrating that $\sup_{\| t - s\| < \delta / r_n}\left| R_n(t) - R_n(s) \right| \inprob 0$, a rate of convergence for $\sup_{\| t - s\| < \delta / r_n}\left| \d{G}_n\left(  \nu_{n,t} - \nu_{n,s}\right) \right|$ is then required.   Defining $\s{F}_{n,\eta} := \{ \nu_{n,t} - \nu_{n,s} : \|t -s \| < \eta\}$ for each $\eta > 0$, this is equivalent to establishing a rate of convergence for the local empirical process $\| \d{G}_n \|_{\s{F}_{n,\delta / r_n}} := \sup_{\xi \in \s{F}_{n,\delta / r_n}} | \d{G}_n \xi|$.  Such rates can be established using tail bounds for empirical processes. We briefly comment on two approaches to obtaining such tail bounds.
 
 
We first  define bracketing and covering numbers of a class of functions $\s{F}$ -- see \cite{van1996weak} for a comprehensive treatment. We denote by $\| F\|_{P,2} = [ P( F^2)]^{1/2}$ the $L_2(P)$ norm of a given $P$-square-integrable function $F:\mathscr{X}\rightarrow\mathbb{R}$. The bracketing number $N_{[]}(\varepsilon, \s{F}, L_2(P))$ of a class of functions $\s{F}$ with respect to the $L_2(P)$ norm is the smallest number of $\varepsilon$-brackets needed to cover $\s{G}$, where an $\varepsilon$-bracket is any set of functions $\{ f: \ell \leq f \leq u\}$ with $\ell$ and $u$ such that $\|\ell-u\|_{P,2} < \varepsilon$. The covering number $N(\varepsilon, \s{F}, L_2(Q))$ of $\s{F}$ with respect to the $L_2(Q)$ norm is the smallest number of $\varepsilon$-balls in $L_2(Q)$ required to cover $\s{F}$. The uniform covering number is the supremum of $N(\varepsilon\|F\|_{2,Q}, \s{F}, L_2(Q))$ over all discrete probability measures $Q$ such that $\|F\|_{Q,2} > 0$, where $F$ is an envelope function for $\s{F}$. The  bracketing and uniform entropy integrals for $\s{F}$ with respect to $F$ are then defined as
\begin{align*}
J_{[]}(\delta, \s{F}) &:= \int_0^{\delta} \left[ 1 + \log N_{[]}\left(\varepsilon \| F\|_{P_0,2}, \s{F}, L_2(P_0)\right) \right]^{1/2}\, d\varepsilon \\
J( \delta, \s{F}) &:= \sup_Q \int_0^{\delta} \left[ 1 + \log N\left(\varepsilon \| F\|_{Q,2}, \s{F}, L_2(Q)\right) \right]^{1/2} \, d\varepsilon \ .
\end{align*}
We discuss two approaches to controlling $\| \d{G}_n \|_{\s{F}_{n,\delta / r_n}}$ using these integrals. Suppose that $\s{F}_{n,\eta}$ has envelope function $F_{n,\eta}$ in the sense that $|\xi(x)| \leq F_{n,\eta}$ for all $\xi \in \s{F}_{n,\eta}$ and $x \in \s{X}$. The first approach is useful when $\| F_{n,\delta / r_n} \|_{P_0, 2}$ can be adequately controlled. Specifically, if either $J(1, \s{F}_{n, \delta/r_n})$ or $J_{[]}(1, \s{F}_{n, \delta/r_n})$ is $\boundeddet(1)$, then $\| \d{G}_n \|_{\s{F}_{n,\delta / r_n}} \leq M_\delta \| F_{n,\delta / r_n} \|_{P_0, 2}$ for all $n$ and some constant $M_\delta \in (0, \infty)$ not depending on $n$ by Theorems 2.14.1 and 2.14.2 of \cite{van1996weak}.

The second approach we consider is useful when the envelope functions do not shrink in expectation, but the functions in $\s{F}_{n,\eta}$ still get smaller in the sense that $\gamma_{n, \delta} := \sup_{\xi \in \s{F}_{n,\delta / r_n}} \|\xi\|_{P_0,2}$ tends to zero. For example, if $\nu_{n,t}$ is defined as $\nu_{n,t}(x) := I(0 \leq x \leq t)$ for each $x \in \s{X} \subseteq \d{R}$, $t \in [0,1]$, and $n$, then $F_{n,\eta} :  x \mapsto I(0 \leq x \leq 1)$ is the natural envelope function for $\s{F}_{n,\eta}$ for all $n$ and $\eta$, so that $\| F_{n,\delta / r_n} \|_{P_0, 2}$ does not tend to zero. However, $\gamma_{n,\delta} \leq \left(\bar{p}_0 \delta / r_n\right)^{1/2}$ if the density $p_0$ corresponding to $P_0$ is bounded above by $\bar{p}_0$, which does tend to zero. In these cases, the basic tail bounds in Theorem 2.14.1 and 2.14.2 of \cite{van1996weak} are too weak. Sharper, but slightly more complicated, bounds may be used instead. Specifically, if $F_{n, \delta / r_n} \leq C < \infty$ for all $n$ large enough and either
\begin{align*}
J\left( \gamma_{n, \delta}, \s{F}_{n,\delta / r_n}\right) + \frac{J\left( \gamma_{n, \delta}, \s{F}_{n,\delta / r_n}\right)^2}{\gamma_{n, \delta}^2 n^{1/2}} \, \text{ or } \, J_{[]}\left( \gamma_{n, \delta}, \s{F}_{n,\delta / r_n}\right) + \frac{J_{[]}\left( \gamma_{n, \delta}, \s{F}_{n,\delta / r_n}\right)^2}{\gamma_{n, \delta}^2 n^{1/2}}
\end{align*}
are $\fasterthandet\left(z_n^{-1} \right)$, then  $\| \d{G}_n \|_{\s{F}_{n,\delta / r_n}} = \fasterthan\left(z_n^{-1}\right)$ by Lemma 3.4.2 of \cite{van1996weak} and Theorem 2.1 of \cite{vandervaart2011}. Analogous statements hold if these expressions are $\boundeddet\left(z_n^{-1} \right)$.

In some cases, both of these approaches must be used to control different terms arising within $R_n(t)$, as for the conditional distribution function discussed in Section~\ref{cond_dist}.

\section{Illustrative examples}

\subsection{Example 1: Estimation of a G-computed distribution function}\label{survival}

We first demonstrate the use of Theorem~\ref{thm:emp_process} in the particular problem in which we wish to draw inference on a G-computed distribution function. Suppose that the data unit is the vector $X=(Y,  A, W)$, where $Y$ is an outcome, $A \in \{0,1\}$ is an exposure, and $W$ is a vector of baseline covariates. The observed data consist of independent draws $X_1,X_2,\ldots,X_n$ from $P_0\in\mathscr{M}$, where $\mathscr{M}$ is a nonparametric model.

For $P\in\mathscr{M}$ and $a_0 \in \{0,1\}$, we define the parameter value $\theta_{P,a_0}$ pointwise as $\theta_{P,a_0}(t) := E_P\left\{ P\left( Y \leq t \mid A = a_0, W\right)\right\}$, the G-computed distribution function of $Y$ evaluated at $t$, where the outer expectation is over the marginal distribution of $W$ under $P$. We are interested in estimating $\theta_{0,a_0} :=\theta_{P_0, a_0}$. This parameter is often of interest as an interpretable marginal summary of the relationship between $Y$ and $A$ accounting for the potential confounding induced by $W$. Under certain causal identification conditions, $\theta_{0, a_0}$ is the distribution function of the counterfactual outcome $Y(a_0)$ defined by the intervention that deterministically sets exposure to $A=a_0$ \citep{robins1986,gill2001}.

For each $t$, the parameter $P\mapsto \theta_{P,a_0}(t)$ is pathwise differentiable in a nonparametric model, and its nonparametric efficient influence function at $P\in\mathscr{M}$ is given by \[\varphi_{P,a_0,t}(y, a, w):=\frac{I(a = a_0)}{g_P(a_0 \mid w)}\left[I(y\leq t)-\bar{Q}_{P}(t  \mid a_0, w)\right]+\bar{Q}_P(t \mid a_0,  w)-\theta_{P,a_0}(t)\ ,\] where $g_P( a_0 \mid w):=P(A=a_0\mid W=w)$ is the propensity score and $\bar{Q}_P(t \mid a_0, w):=P\left(Y\leq t \mid A=a_0,W=w\right)$ is the conditional exposure-specific distribution function, as implied by $P$ \citep{van2003unified}. Given estimators $g_n$ and $\bar{Q}_n$ of $g_0:=g_{P_0}$ and $\bar{Q}_0:=\bar{Q}_{P_0}$, respectively, several approaches can be used to construct, for each $t$, an asymptotically linear estimator of $\theta_0(t)$ with influence function $\phi_{0,a_0,t}=\varphi_{P_0,a_0,t}$. For example, the use of either optimal estimating equations or the one-step correction procedure leads to the doubly-robust augmented inverse-probability-of-weighting estimator \[\theta_{n,a_0}(t) := \frac{1}{n}\sum_{i=1}^{n}\frac{I(A_i = a_0)}{g_n(a_0 \mid W_i)}\left[I(Y_i\leq t)-\bar{Q}_n(t \mid a_0, W_i)\right]+\frac{1}{n}\sum_{i=1}^{n}\bar{Q}_n(t \mid a_0, W_i)\] as discussed in detail in \cite{van2003unified}. Under conditions on $g_n$ and $\bar{Q}_n$, including consistency at fast enough rates, $\theta_{n,a_0}(t)$ is asymptotically efficient relative to $\mathscr{M}$. In this case, $\theta_{n,a_0}(t)$ satisfies \eqref{asy_linear} with influence function $\phi_{0,a_0,t}$. However, there is no guarantee that $\theta_{n,a_0}$ is monotone.

In the context of this example, we can identify simple sufficient conditions under which conditions (A)--(B), and hence the asymptotic equivalence of the initial and isotonized estimators of the G-computed distribution function, are guaranteed. Specifically, we find this to be the case when: \begin{enumerate}[\ \ \ \ (i)]
\item there exists some $\eta>0$ such that $g_0(a_0 \mid W) \geq \eta$ almost surely under $P_0$, and;
\item there exist non-negative real-valued functions $K_1,K_2$ such that \[K_1(w)|t-s|\ \leq\ |\bar{Q}_0(t \mid a_0, w) - \bar{Q}_0(s \mid a_0, w)|\ \leq\ K_2(w)|t - s|\] for all $t, s \in \s{T}$, and such that, under $P_0$, $K_1(W)$ is strictly positive with non-zero probability and $K_2(W)$ has finite second moment.
\end{enumerate} 

We conducted a simulation study to validate our theoretical results in the context of this particular example. For samples sizes $n\in\{100, 250, 500, 750, 1000\}$, we generated $1000$ random datasets as follows. We first simulated a bivariate covariate $W$ with independent components $W_1$ and $W_2$, respectively distributed as a Bernoulli variate with success probability $0.5$ and a uniform variate on $(-1,1)$. Given $W=(w_1,w_2)$, exposure $A$ was simulated from a logistic regression model with $P_0(A = 1 \mid W_1=w_1,W_2=w_2) = \text{expit}(0.5 +  w_1- 2 w_2)$. Given $W=(w_1,w_2)$ and $A=a$, $Y$ was simulated as the inverse-logistic transformation of a normal variate with mean $0.2 -0.3 a-4w_2$ and variance $0.3$.

For each simulated dataset, we estimated $\theta_{0,0}(t)$ and $\theta_{0,1}(t)$ for $t$ equal to each outcome value observed between $0.1$ and $0.9$. To do so, we used the estimator described above, with propensity score and conditional exposure-specific distribution function estimated using correctly-specified parametric models. We employed two correction procedures for the estimators $\theta_{n,0}$ and $\theta_{n,1}$. First, we projected $\theta_{n,0}$ and $\theta_{n,1}$ onto the space of monotone functions separately. Second, noting that $\theta_{0,0}(t) \leq \theta_{0,1}(t)$ for all $t$, so that $(a,t)\mapsto\theta_{0,a}(t)$ is component-wise monotone for this particular data-generating distribution, we considered the projection of $(a, t) \mapsto \theta_{n,a}(t)$ onto the space of bivariate monotone functions on $\{0,1\} \times \s{T}$. For each simulation and each projection procedure, we recorded the maximal absolute differences between (i) the initial and and projected estimates, (ii) the initial estimate and the truth, and (iii) the projected estimate and the truth. We also recorded the maximal widths of the initial and projected confidence bands.

Figure~\ref{fig:sim_results} displays the results of this simulation study, with output from the univariate and bivariate projection approaches summarized in the top and bottom rows, respectively. The left column displays the empirical distribution of the scaled maximum absolute discrepancy between $\theta_n$ and $\theta_n^*$ for all sample sizes studied. This plot confirms that the discrepancy between these two estimators indeed decreases faster than $n^{-1/2}$, as our theory suggests. Furthermore, for each $n$, the discrepancy is larger for the two-dimensional projection.

The middle column of Figure~\ref{fig:sim_results} displays the empirical distribution function of the ratio between the maximum discrepancy between $\theta_n$ and $\theta_0$ and that of $\theta_n^*$ and $\theta_0$. This plot confirms that $\theta_n^*$ is always at least as close to $\theta_0$ than is $\theta_n$ over $\s{T}_n$.  The maximum discrepancy between $\theta_n$ and $\theta_0$ can be more than 25\% larger than that between $\theta_n^*$ and $\theta_0$ in the univariate case, and up to 50 \% larger in the bivariate case. 

The right column of Figure~\ref{fig:sim_results} displays the empirical distribution function of the ratio between the maximum size of the initial uniform 95\% influence function-based confidence band and that of the isotonic band. For large samples, the maximal widths are often close, but for smaller samples, the initial confidence bands can be up to 50\% larger than the isotonic bands, especially for the bivariate case. The empirical coverage of both bands is provided in Table~\ref{tab:coverages}. The coverage of the isotonic band is essentially the same as the initial band for the univariate case, whereas it  is slightly larger than that of the initial band in the bivariate case.

\begin{figure}[ht]
\centering
\includegraphics[width=\linewidth]{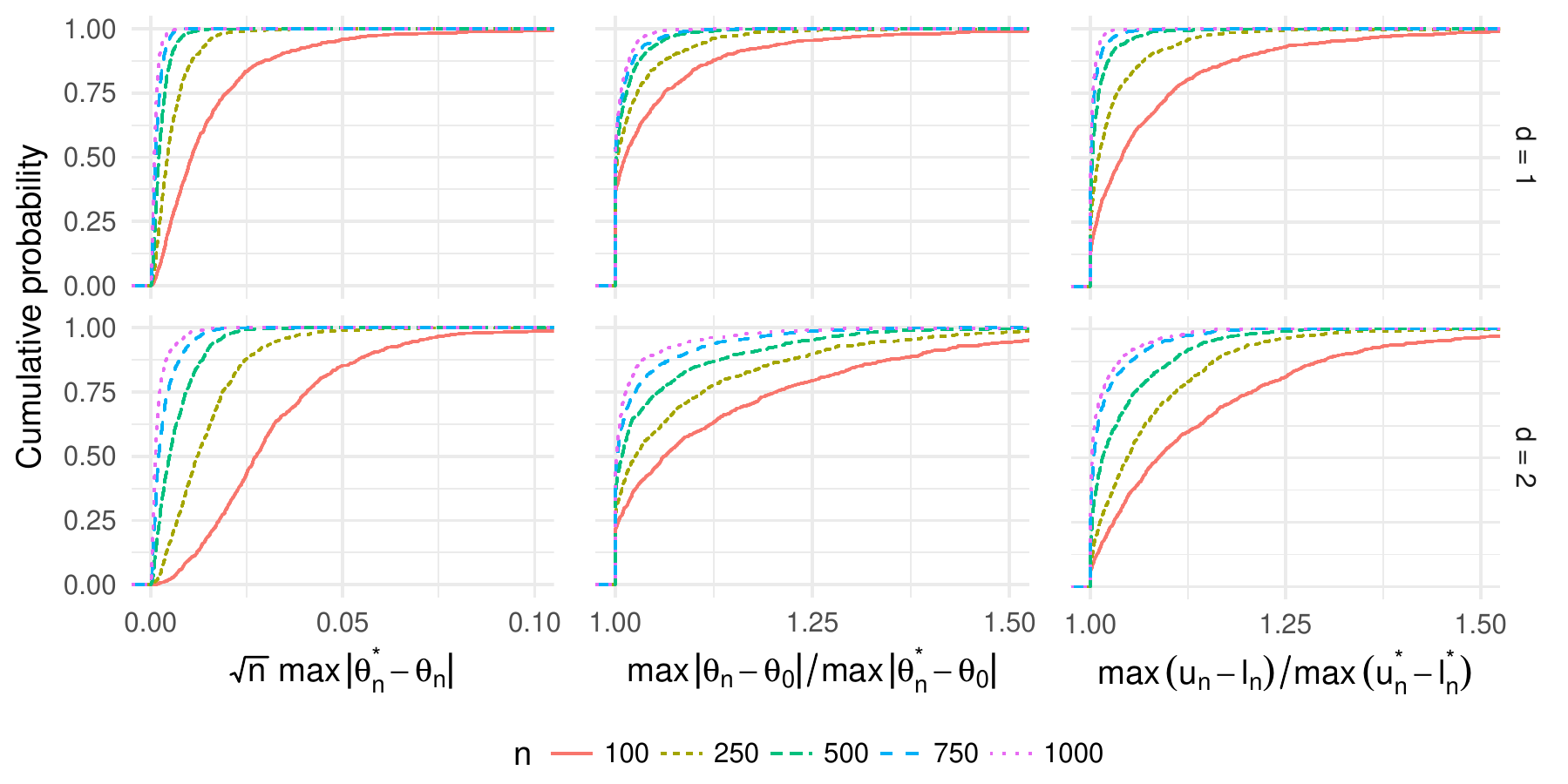}
\caption{Summary of simulation results for G-computed distribution function. Each plot shows cumulative distributions of a particular discrepancy over 1000 simulated datasets for different values of $n$. Left panel: maximal absolute difference between the initial and isotonic estimators over the grid used for projecting, scaled up by root-$n$. Middle panel: ratio of the maximal absolute difference between the initial estimator and the truth and the maximal absolute difference between the isotonic estimator and the truth. Right panel: ratio of the maximal width of the initial confidence band and the maximal width of the isotonic confidence band. The top row shows the results for the univariate projection, and the bottom row shows the results for the bivariate projection.}
\label{fig:sim_results}
\end{figure}

\begin{table}
\caption{Coverage of 95\% confidence bands for the true counterfactual distribution function.}
\vspace*{1em}
\label{tab:coverages}
\centering
\begin{tabular}{crrrrrr}

&$n$ & 100 & 250 & 500 & 750 & 1000 \\
\hline
\multirow{2}{*}{d=1}&Initial band & 92.5 & 94.1& 96.0  & 94.5 & 95.5 \\
&Monotone band &  92.5 & 94.1& 96.0  & 94.5 & 95.5 \\
\hline
\multirow{2}{*}{d=2}&Initial band &  93.9 &94.0 & 95.0 & 94.6 & 94.9 \\
&Monotone band & 95.7& 95.9& 95.5 &95.3& 95.1 \\
\hline
\end{tabular}
\end{table}

\subsection{Example 2: Estimation of a conditional distribution function}\label{cond_dist}

We next demonstrate the use of Theorem~\ref{monotone_supnorm} with dimension $d=2$ for drawing inference on a conditional distribution function. Suppose that the data unit is the vector $X = (A, Y)$, where $Y$ is an outcome and $A$ is now a continuous exposure. The observed data consist of independent draws $(A_1, Y_1),(A_2,Y_2),\ldots,(A_n,Y_n)$ from $P_0\in\mathscr{M}$, where $\mathscr{M}$ is a nonparametric model. We define the parameter value $\theta_P$ pointwise as $\theta_{P}(t_1, t_2) := P\left( Y \leq t_1 \mid A = t_2 \right)$. Thus, $\theta_P$ is the conditional distribution function of $Y$ at $t_1$ given $A = t_2$. The map $(t_1, t_2) \mapsto \theta_P(t_1, t_2)$ is necessarily monotone in $t_1$ for each fixed $t_2$, and in some settings, it may be known that it is also monotone in $t_2$ for each fixed $t_1$. This parameter completely describes the conditional distribution of $Y$ given $A$, and can be used to obtain the conditional mean, conditional quantiles, or any other conditional parameter of interest. 

For each $t_1$, the true function $\theta_0(t_1, t_2) = \theta_{P_0}(t_1, t_2)$ may be written as the conditional mean of $I(Y \leq t_1)$ given $A = t_2$. Hence, any method of nonparametric regression can be used to estimate $t_2\mapsto \theta_0(t_1, t_2)$ for fixed $t_1$, and repeating such a method over a grid of values of $t_1$ yields an estimator of the entire function. We expect that our results would apply to many of these methods. Here, we consider the local linear estimator \citep{fan1996local}, which may be expressed as
\[ \theta_n(t_1, t_2) :=  \frac{1}{nh_n} \sum_{i=1}^n I(Y_i \leq t_1)\left[ \frac{s_{2,n}(t_2)  - s_{1,n}(t_2) \left(A_i - t_2\right)}{s_{0,n}(t_2)s_{2,n}(t_2) - s_{1,n}(t_2)^2}\right] K\left( \frac{A_i - t_2}{h_n}\right) \ ,\]
where $K:\d{R}\to\d{R}$ is a symmetric and bounded kernel function, $h_n \to 0$ is a sequence of bandwidths, and $s_{j,n}(t_2) :=  \frac{1}{nh_n} \sum_{i=1}^n \left(A_i - t_2\right)^j K\left( \frac{A_i - t_2}{h_n}\right)$ for $j \in \{0,1,2\}$. Under regularity conditions on the true distribution function $\theta_0$, the marginal density $f_0$ of $A$, the bandwidth sequence $h_n$, and the kernel function $K$, for any fixed $(t_1, t_2)$, $\theta_n$ satisfies
\[ (nh_n)^{1/2} \left[ \theta_n(t_1, t_2) - \theta_0(t_1, t_2) - h_n^2 V_K  b_0(t_1, t_2)\right] \indist N\left(0, S_K  v_0(t_1, t_2)\right), \]
where $V_K := \int x^2K(x)  dx$ is the variance of $K$, $S_K := \int K(x)^2  dx$, and $b_0(t_1, t_2)$ and $v_0(t_1, t_2)$ depend on the derivatives of $\theta_0$ and on $f_0$. If $h_n$ is chosen to be of order $n^{-1/5}$, the rate that minimizes the asymptotic mean integrated squared error of $\theta_n$ relative to $\theta_0$, then $n^{2/5} \left[ \theta_n(t_1, t_2) - \theta_0(t_1, t_2) \right]$ converges in law to a normal random variate with mean $V_K  b_0(t_1, t_2)$ and variance $S_K  v_0(t_1, t_2)$.  Under stronger regularity conditions, the rate of convergence of the uniform norm $\|\theta_n - \theta_0\|_\s{T}$ can be shown to be $(n h_n / \log n)^{1/2}$ \citep{hardle1988}.

Theorem~\ref{thm:emp_process} cannot be used to establish (A) in this problem, since $\theta_n$ is not an asymptotically linear estimator. Furthermore, as discussed above, recent results suggest that $\{ r_n[\theta_n(t) - \theta_0(t)] : t \in \s{T}\}$ does not converge weakly to a tight limit in $\ell^{\infty}(\s{T})$ for any useful rate $r_n$. Despite this lack of weak convergence, condition (A) can be verified directly in the context of this example under smoothness conditions on $\theta_0$ and $f_0$ using the tail bounds for empirical processes outlined in Section~\ref{sec:kern}. Denoting by $\theta_{0,t_2}'$ and $\theta_{0,t_2}''$ the first and second derivatives of $\theta_0$ with respect to its second argument, we define
\begin{align*}
 R_{\theta}^{(2)}(t, \delta) &:= \theta_0(t_1, t_2 + \delta) - \theta_0(t_1, t_2) - \delta\theta_{0,t_2}'(t_1, t_2) - \tfrac{1}{2} \delta^2\theta_{0,t_2}''(t_1, t_2)\ ,
 \end{align*}
 and $R_{f}^{(1)}(t, \delta) := f_0(t_2 +\delta ) - f_0(t_2) - \delta f_0'(t_2)$, where $f_0'$ is the derivative of $f_0$. We then introduce the following conditions on $\theta_0$, $f_0$, and $K$:
 \begin{description}
\item[(d)] $\theta_{0,t_2}''$ exists and is continuous on $\s{T}$, and as $\delta \to 0$, $\sup_{t \in \s{T}} | R_{\theta}^{(2)}(t, \delta)| = \fasterthandet(\delta^2)$;
\item[(e)] $\inf_{t\in \s{T}} f_0(t_2) > 0$, $f_0'$ exists and is continuous on $\s{T}$, and $\sup_{t \in \s{T}}| R_{f}^{(1)}(t, \delta)|=  \fasterthandet(\delta)$;
\item[(f)] $K$ is a Lipschitz function supported on $[-1,1]$ and satisfies condition (M) of \cite{stupfler2016}.
\end{description} 
We also define $\nu_{n,t}(y,a) :=  \left[ I(y \leq t_1) - \theta_0(t_1, a)\right]  K\left( \frac{a - t_2}{h_n}\right)$, $g_n(t_2) := s_{0,n}(t_2)s_{2,n}(t_2) - s_{1,n}(t_2)^2$, and $R_n(t) := h_n^{-1/2} \left[ \frac{ s_{2,n}(t_2)}{g_n(t_2)} \d{G}_n \nu_{n,t} -  \frac{ s_{1,n}(t_2)}{g_n(t_2)} \d{G}_n \left( \ell_t \nu_{n,t}\right)\right]$. We then have the following result. 
\begin{proposition}\label{prop:kern_unif}
Suppose conditions (d)-(f) hold, $nh_n^5 =\boundeddet(1)$, and $nh_n^4/ \log h_n^{-1} \longrightarrow \infty$. Then
\[ \sup_{t \in \s{T}} \left| \left(nh_n \right)^{1/2} \left[ \theta_n(t_1, t_2) - \theta_0(t_1, t_2) \right] - \left(n h_n^5\right)^{1/2} \tfrac{1}{2}\theta_{0, t_2}''(t_1, t_2)  K_2 - R_n(t)  \right| \inprob 0 \ .\]
\end{proposition}
Proposition~\ref{prop:kern_unif} aids in establishing the following result, which formally establishes asymptotic equivalence of the local linear estimator of a conditional distribution function and its correction obtained via isotonic regression at the rate $r_n = (n h_n)^{1/2}$.
 \begin{proposition}\label{prop:cond_dist}
Suppose conditions (d)-(f) hold and $nh_n^5 \longrightarrow c \in (0, \infty)$. Then condition (A) holds for the local linear estimator with $r_n = (nh_n)^{1/2}$. 
\end{proposition}
The proof of Propositions~\ref{prop:kern_unif} and~\ref{prop:cond_dist} are provided in Supplementary Material. These results may be of interest in their own right for establishing other properties of the local linear estimator.

As with the first example, we conducted a simulation study to validate our theoretical results. For samples sizes $n\in\{100, 250, 500, 750, 1000\}$, we generated $1000$ random datasets as follows. We first simulated $A$ as a Beta$(2,3)$ variate. Given $A=a$, $Y$ was simulated as the inverse-logistic transformation of a normal variate with mean $0.5 \times [1 + (a - 1.2)^2]$ and variance one.

For each simulated dataset, we estimated $\theta_0(y,a)$ for each $(y,a)$ in an equally spaced square grid of mesh $\omega_n = n^{-4/5}$. For each unique $y$ in this grid, we estimated the function $a\mapsto \theta_0(y, a)$ using the local linear estimator, as implemented in the \texttt{R} package \texttt{KernSmooth} \citep{KernSmooth, wand1995kernsmooth}. For each value of $y$ in the grid, we computed the optimal bandwidth based on the direct plug-in methodology of \cite{ruppert1995effective} as implemented by the \texttt{dpill} function, and we then set our bandwidth as the average of these $y$-specific bandwidths. We constructed initial confidence bands using a variable-width nonparametric bootstrap \citep{hall2001bootstrap}.

We first note that, for all sample sizes considered, over 99\% of simulations had monotonicity violations in both the $y$- and $a$-directions. Figure~\ref{fig:sim_results_cond_dist} displays the results of this simulation study. The left exhibit of Figure~\ref{fig:sim_results_cond_dist} confirms that the discrepancy between $\theta_n$ and $\theta_n^*$ decreases faster than $r_n^{-1}=n^{-2/5}$, as our theory suggests. The middle exhibit indicates that in roughly 50\% of simulations, there is less than 5\% difference between  $\| \theta_n^* - \theta_0\|_{\s{T}_n}$ and $\|\theta_n - \theta_0\|_{\s{T}_n}$, but even for $n=1000$, in roughly 25\% of simulations, $\theta_n^*$ offers at least a 25\% improvement in estimation error. In smaller samples, the estimation error of $\theta_n^*$ is less than half that of $\theta_n$ in 5-10\% of simulations. The rightmost exhibit indicates that the projected confidence bands regularly reduce the uniform size of the initial bands by 10-20\%. Finally, the empirical coverage of uniform 95\% bootstrap-based bands and their projected versions is provided in Table~\ref{tab:coverages2}. As before, the projected band is always more conservative than the initial band, and the difference in coverage diminishes as $n$ grows. However, the initial bands in this example are anti-conservative, even at $n=1000$, likely due to the slower rate of convergence, and the corrected bands offer a much more substantial improvement in this example than in the first.

\begin{figure}[ht]
\centering
\includegraphics[width=\linewidth]{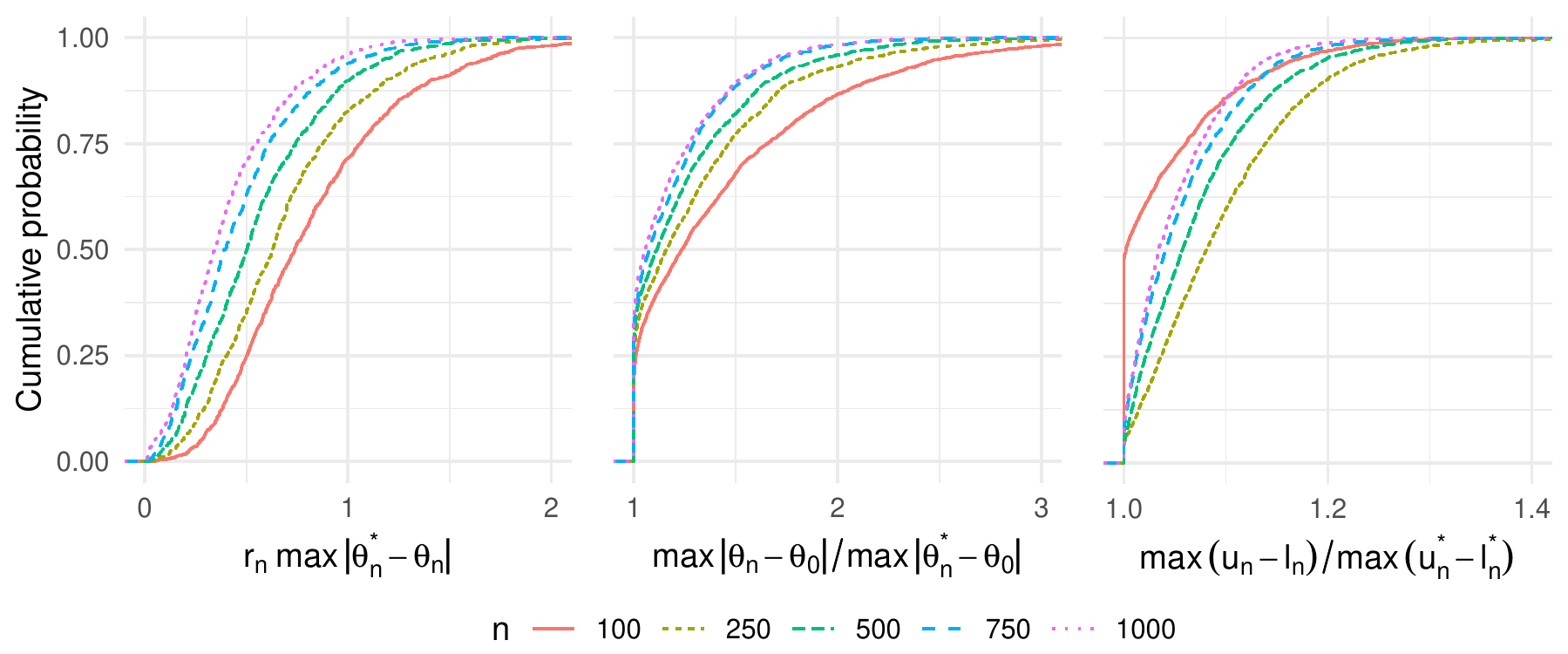}
\caption{Summary of simulation results for conditional distribution function. The three columns display the same results as those in Figure~\ref{fig:sim_results}.}
\label{fig:sim_results_cond_dist}
\end{figure}

\begin{table}[h]
\caption{Coverage of 95\% confidence bands for the true conditional distribution function.}
\vspace*{1em}
\label{tab:coverages2}
\centering
\begin{tabular}{rrrrrr}
\hline
$n$ & 100 &  250 & 500 & 750 & 1000 \\
\hline
Initial band  & 37.6 & 64.9 & 83.2 & 86.3 & 89.7\\ 
Monotone band & 60.8 & 80.4 & 90.3 & 92.3 & 93.9  \\ 
\hline
\end{tabular}
\end{table}

\section{Discussion}\label{conclusion}

Many estimators of function-valued parameters in nonparametric and semiparametric models are not guaranteed to respect shape constraints on the true function. A simple and general solution to this problem is to project the initial estimator onto the constrained parameter space over a grid whose mesh goes to zero fast enough with sample size. However, this introduces the possibility that the projected estimator has different properties than the original estimator.  In this paper, we studied the important shape constraint of multivariate component-wise monotonicity. We provided results indicating that the projected estimator is generically no worse than the initial estimator, and that if the true function is strictly increasing and the initial estimator possesses a relatively weak type of stochastic equicontinuity, the projected estimator is uniformly asymptotically equivalent to the initial estimator. We provided especially simple sufficient conditions for this latter result when the initial estimator is uniformly asymptotically linear, and provided guidance on establishing the key condition for kernel smoothed estimators.

We studied the application of our results in two examples: estimation of a G-computed distribution function, for use in understanding the effect of a binary exposure on an outcome when the exposure-outcome relationship is confounded by recorded covariates, and of a conditional distribution function, for use in characterizing the marginal dependence of an outcome on a continuous exposure. In numerical studies, we found that the projected estimator yielded  improvements over the initial estimator. The improvements were especially strong in the latter example.

In our examples, we only studied corrections in dimensions $d=1$ and $d=2$. In future work, it would be interesting to consider corrections in dimensions higher than 2. For example, for the conditional distribution function, it would be of interest to study multivariate local linear estimators for a continuous exposure $A$ taking values in $\d{R}^{d-1}$ for $d > 2$. Since tailored algorithms for computing the isotonic regression do not yet exist for $d > 2$, it would also be of interest to determine whether a version of Theorem~\ref{monotone_supnorm} could be established for the relaxed isotonic estimator proposed by \cite{fokianos2017integrated}. Alternatively, it is possible that the uniform stochastic equicontinuity currently required by \cite{chernozhukov2010quantile} and \cite{daouia2012isotonic} for asymptotic equivalence of the rearrangement- and envelope-based corrections, respectively, could be relaxed along the lines of our condition (A). Finally, our theoretical results do not give the exact asymptotic behavior of the projected estimator or projected confidence band when the true function possesses flat regions. This is also an interesting topic for future research.

\vspace{.2in}
\singlespacing
{\footnotesize
\section*{Acknowledgements}
The authors gratefully acknowledge support from the Career Development Fund of the Department of Biostatistics at the University of Washington (MC) and from NIAID grants 5UM1AI058635 (TW, MC) and 5R01AI074345 (MJvdL).
}

\doublespacing
\appendix

\section{Proof of Theorem~\ref{thm:barlow}}

Part (i) follows from Corollary B to Theorem 1.6.1 of \cite{robertson1988order}. For parts (ii) and (iii), we note that by assumption
\[ |\theta_n^*(t) - \theta_0(t)|\ \leq\ \sum_k \lambda_{k,n}(t) |\theta_n^*(s_k) - \theta_0(s_k)| + \sum_k \lambda_{k,n}(t) |\theta_0(s_k) - \theta_0(t)|\] for every $t \in \s{T}$, 
where $\sum_k \lambda_{k,n}(t) = 1$, and for each $k$, $s_k \in \s{T}_n$ and $\| s_k - t\| \leq 2\omega_n$. By part (i), the first term is bounded above by $\sup_{s \in \s{T}_n} | \theta_n(s) - \theta_0(s)|$. The second term is bounded above by $\gamma(2\omega_n)$, where we define $\gamma(\delta) := \sup \{ |\theta_0(t) - \theta_0(s)| : t, s \in \s{T}, \|t - s\| \leq \delta\}$.  If $\theta_0$ is continuous on $\s{T}$, then it is also uniformly continuous since $\s{T}$ is compact. Therefore, $\gamma(\delta) \to \gamma(0)=0$ as $\delta \to 0$, so that $\gamma(2\omega_n) \inproblow 0$ if $\omega_n \inproblow 0$. If $\gamma(\delta) =\fasterthandet(\delta^{\alpha})$ as $\delta \to 0$, then $\gamma(2\omega_n) = \fasterthan(\omega_n^{\alpha})$.

Part (iv) follows from the proof of Proposition 3 of \cite{chernozhukov2009rearrangement}, which applies  to any order-preserving monotonization procedure. For the first statement of (v), by their definition as minimizers of the least-squares criterion function, we note that $\sum_{t \in \s{T}_n} u_n^*(t) = \sum_{t \in \s{T}_n} u_n(t)$, and similarly for $\ell_n^*$. The second statement of (v) follows from a slight modification of Theorem 1.6.1 of \cite{robertson1988order}. As stated, the result says that $\sum_{t \in \s{T}_n} G( \theta^*(t) -  \theta(t) ) \leq \sum_{t \in \s{T}_n} G(\theta(t) -  \psi(t))$ for any convex function $G : \d{R} \to \d{R}$ and monotone function $\psi$, where $\theta^*$ is the isotonic regression of $\theta$ over $\s{T}_n$. A straightforward adaptation of the proof indicates that $\sum_{t \in \s{T}_n} G( \theta_1^*(t) -  \theta_2^*(t) ) \leq \sum_{t \in \s{T}_n} G(\theta_1(t) -  \theta_2(t))$, where now $\theta_1^*$ and $\theta_2^*$ are the isotonic regressions of $\theta_1$ and $\theta_2$ over $\s{T}_n$, respectively. As in Corollary B, taking $G(x) = |x|^p$ and letting $p \to\infty$ yields that $\| \theta_1^* - \theta_2^*\|_{\s{T}_n} \leq \|\theta_1 - \theta_2\|_{\s{T}_n}$. Applying this with $\theta_1 = u_n$ and $\theta_2 = \ell_n$ establishes the second portion of (v). \qed

\section{Proof of Theorem~\ref{monotone_supnorm}}

We first prove Lemmas~\ref{lemma:moduli_of_continuity},~\ref{neighborhoods}, and~\ref{pava}.

\begin{proof}[\bfseries Proof of Lemma~\ref{lemma:moduli_of_continuity}] In view of the triangle inequality, we note that $\left| \theta_n(t) - \theta_n(s)\right|$ is bounded above by $\left| \{\theta_n(t) - \theta_0(t)\} - \{\theta_n(s) - \theta_0(s)\}\right| + \left|\theta_0(t) - \theta_0(s) \right|$. The first term is $\fasterthan(r_n^{-1})$ by (A), whereas the second term is $\fasterthan(r_n^{-1})$ by (B).\end{proof}

\begin{proof}[\bfseries Proof of Lemma~\ref{neighborhoods}]
Let $\epsilon > 0$ and $\eta_n := \epsilon/r_n$. Suppose that $\kappa_n > \eta_n$. Then, there exist $s,t \in \s{T}$ with $s < t$ and $\|t -s\| > \eta_n$ such that $\theta_n(s) \geq \theta_n(t)$. We claim that there must also exist $s^*, t^* \in \s{T}$ with $s^* < t^*$ and $\|t^* - s^*\| \in [\eta_n / 2, \eta_n]$ such that $\theta_n(s^*) \geq \theta_n(t^*)$. To see this, let $J = \lfloor \|t - s\| / (\eta_n / 2)\rfloor - 1$, and note that $J \geq 1$. Define $t_j := s + (j\eta_n/2)(t-s) / \|t-s\|$ for $j = 0,1, \dotsc, J$, and set $t_{J+1} := t$. Thus, $t_j < t_{j+1}$ and $\|t_{j+1} - t_j\| \in [\eta_n / 2, \eta_n]$ for each $j = 0,1, \dotsc, J$. Since then $\sum_{j=0}^J [\theta_n(t_{j+1}) - \theta_n(t_{j})] = \theta_n(t) - \theta_n(s) \leq 0$, it must be that $\theta_n(t_{j+1}) \leq \theta_n(t_{j})$ for at least one $j$. This proves the claim.

We now have that $\kappa_n > \eta_n$ implies that there exist $s, t \in \s{T}$ with $s < t$ and $\|t - s\| \in [\eta_n / 2, \eta_n]$ such that $\theta_n(s) \geq \theta_n(t)$.  This further implies that 
\[ \{\theta_n(t) - \theta_0(t)\} - \{\theta_n(s) - \theta_0(s)\} \leq -\{\theta_0(t) - \theta_0(s)\} \leq -K_0\|t-s\| \leq -K_0 \eta_n /2\]
by condition (B). Finally, this allows us to write
\begin{align*}
P_0\left(\kappa_n >  \epsilon /r_n\right)\ \leq\ P_0\left( \sup_{\|t -s\| \leq \epsilon / r_n} \left| r_n[\theta_n(t) - \theta_0(t)] - r_n[\theta_n(s) - \theta_0(s)] \right|  \geq K_0\epsilon/2\right).
\end{align*}
By condition (A), this probability tends to zero for every $\epsilon > 0$, which completes the proof.\end{proof}

\begin{proof}[\bfseries Proof of Lemma~\ref{pava}]
By Theorem 1.4.4 of \cite{robertson1988order}, for any $t \in \s{T}_n$,
\[ \theta_n^*(t) = \max_{U \in \s{U}_t} \min_{L \in \s{L}_t} \theta_n(U \cap L)=  \min_{L \in \s{L}_t} \max_{U \in \s{U}_t}\theta_n(U \cap L),\]
where, for any finite set $S\subseteq \s{T}_n$,  $\theta_n(S)$ is defined as $|S|^{-1} \sum_{s \in S} \theta_n(s)$. The sets $U$ range over the collection $\s{U}_t$ of upper sets of $\s{T}_n$ containing $t$, where $U\subseteq \s{T}_n$ is called an upper set if $t_1 \in U, t_2 \in \s{T}_n$ and $t_1 \leq t_2$ implies $t_2 \in U$.  The sets $L$ range over the collection $\s{L}_t$ of lower sets of $\s{T}_n$ containing $t$, where $L\subseteq \s{T}_n$ is called a lower set if $t_1 \in L, t_2 \in \s{T}_n$ and $t_2 \leq t_1$ implies $t_2 \in L$. 

Let $U_t := \{ s : s \geq t\}$ and $L_t := \{s : s \leq t\}$. First, suppose there exists $L_0\in\s{L}_t$ and $s_0 \in L_0$ with $s_0>t$ and $\|t - s_0\| > \kappa_n$. Then, we claim that there exists another lower set $L_0'\in \s{L}_t$ such that $\theta_n( U_t \cap L_0) > \theta_n( U_t \cap L_0')$. If $\theta_n(U_t \cap L_0) > \theta_n(t) = \theta_n(U_t \cap L_t)$, then $L_0' = L_t$ satisfies the claim. Otherwise, if $\theta_n( U_t \cap L_0) \leq \theta_n(t)$, let $L_0' := L_0 \setminus \{ s : s > t, \| t - s\| > \kappa_n\}$. One can verify that $L_0'\in \s{L}_t$, and since $s_0 \in L_0\setminus L_0'$, $L_0'$ is a strict subset of $L_0$. Furthermore, by definition of $\kappa_n$, $\theta_n(s) > \theta_n(t)$ for all $s > t$ such that $\| t- s\| > \kappa_n$, and since $\theta_n( U_t \cap L_0) \leq \theta_n(t)$, removing these elements from $L_0$ can only reduce the average, so that $\theta_n(U_t \cap L_0') < \theta_n( U_t \cap L_0)$. This establishes the claim. By an analogous argument, we can show that if there exists $U_0\in\s{U}_t$ and $s_0\in U_0$ with $s_0 < t$ and $\|t - s_0\| > \kappa_n$, then there exists another upper set $U_0'\in \s{U}_t$ such that $\theta_n(U_0 \cap L_t) < \theta_n(U_0' \cap L_t)$.

Let $L^* \in \argmin_{L \in \s{L}_t} \theta_n(U_t \cap L)$ and $U^* \in \argmax_{U \in \s{U}_t} \theta_n(U \cap L_t)$. Then
\begin{align*}
\theta_n^*(t) = \max_{U \in \s{U}_t} \min_{L\in\s{L}_t} \theta_n (U \cap L)\ &\geq\ \min_{L \in \s{L}_t} \theta_n(U_t \cap L) = \theta_n(U_t \cap L^*)\\
\theta_n^*(t) = \min_{L \in \s{L}_t} \max_{U \in \s{U}_t} \theta_n(U \cap L)\ &\leq\ \max_{U \in \s{U}_t} \theta_n(U \cap L_t) = \theta_n(U^* \cap L_t)\ .
\end{align*}
Hence, $\theta_n(U_t \cap L^*) \leq \theta_n^*(t) \leq \theta_n(U^* \cap L_t)$. By the above argument, $\theta_n(U_t \cap L^*) \geq \inf\{ \theta_n(s) : s \geq t, \| t - s\| \leq \kappa_n\}$ and $\theta_n(U^* \cap L_t) \leq \sup\{ \theta_n(s) : s \leq t, \|t - s\| \leq \kappa_n\}$. Therefore,
\[ \inf\{ \theta_n(s) - \theta_n(t) : \| t - s\| \leq \kappa_n\} \leq \theta_n^*(t) - \theta_n(t) \leq \sup\{ \theta_n(s) - \theta_n(t) : \|t - s\| \leq \kappa_n\}\ ,\]
and thus, $|\theta_n^*(t) - \theta_n(t) | \leq \sup\{ |\theta_n(s) - \theta_n(t)| : \|t - s\| \leq \kappa_n\}$. Taking the maximum over $t \in \s{T}_n$ yields the claim.\end{proof}

The proof of Theorem~\ref{monotone_supnorm} follows easily from Lemmas~\ref{lemma:moduli_of_continuity},~\ref{neighborhoods}, and~\ref{pava}.
\begin{proof}[\bfseries Proof of Theorem~\ref{monotone_supnorm}]
By construction, for each $t \in \s{T}$, we can write 
\[|\theta_n^*(t) - \theta_n(t)|\ \leq\ \Sigma_{j=1}^{2^d} \lambda_{j,n}(t) | \theta_n^*(s_j) - \theta_n(s_j)|  + \Sigma_{j=1}^{2^d} \lambda_{j,n}(t) | \theta_n(s_j) - \theta_n(t)|\ ,\]
where $s_j \in \s{T}_n$ and $\| s_j - t\| \leq 2\omega_n$ for all $t,s_j$ by definition. Thus, since $\sum_j  \lambda_{j,n}(t) = 1$,
\[ \sup_{t \in \s{T}} |\theta_n^*(t) - \theta_n(t)|\ \leq\ \max_{t \in \s{T}_n}  |\theta_n^*(t) - \theta_n(t)| + \sup_{\| s - t\| \leq 2\omega_n}  | \theta_n(s) - \theta_n(t)|\ .\]
By Lemma~\ref{pava}, the first summand is bounded above by $\sup_{\|s- t\| \leq \kappa_n} |\theta_n(s) - \theta_n(t)|$,  which is $\fasterthan(r_n^{-1})$ by Lemmas~\ref{lemma:moduli_of_continuity} and~\ref{neighborhoods}. The second summand is $\fasterthan(r_n^{-1})$ by Lemma~\ref{lemma:moduli_of_continuity}. \end{proof}

\section{Proof of Corollary~\ref{cor:band}}\label{app:cor1}

We note that $\ell_n(t) \leq \theta_0(t) \leq u_n(t)$ if and only if
\begin{align*}
\left\{ r_n[\theta_n(t) - \ell_n(t)] - \gamma_\alpha(t) \right\} + \gamma_\alpha(t)&\geq r_n [\theta_n(t) - \theta_0(t) ]\\
&\geq -\gamma_\alpha(t) - \left\{ r_n[u_n(t) - \theta_n(t)] - \gamma_\alpha(t) \right\}.
\end{align*}
Therefore, by conditions (a)--(c), $P_0\left(\ell_n(t) \leq \theta_0(t) \leq u_n(t)\mbox{ for all }t\in\s{T}\right) \to 1-\alpha$. Next, we let $\delta > 0$ and note that $\sup_{\|t -s\| \leq \delta / r_n} \left| r_n\{ \ell_n(t) - \theta_0(t)\} - r_n\{ \ell_n(s) - \theta_0(s) \}\right|$ is bounded above by
\begin{align*}&\sup_{\|t -s\| \leq \delta / r_n} \left| r_n\{ \theta_n(t) - \theta_0(t)\} - r_n\{ \theta_n(s) - \theta_0(s) \} \right|+ 2\| r_n(\theta_n - \ell_n) - \gamma_\alpha\|_{\s{T}}\\
&\qquad+\sup_{\|t -s\| \leq \delta / r_n}| \gamma_\alpha(t) - \gamma_\alpha(s)|.
\end{align*}
The first term tends to zero in probability by (A), the second by conditions (a)--(c), and the third by the assumed uniform continuity of $\gamma_\alpha$. An analogous decomposition holds for $u_n$. Therefore, we can apply Theorem~\ref{monotone_supnorm} with $u_n$ and $\ell_n$ in place of $\theta_n$ to find that $\|\ell_n^* - \ell_n\|_{\s{T}} = \fasterthan(r_n^{-1})$ and $\|u_n^* - u_n\|_{\s{T}} = \fasterthan(r_n^{-1})$. Finally, applying an analogous argument to the event $\ell_n^* \leq \theta_0 \leq u_n^*$ as we applied to $\ell_n \leq \theta_0 \leq u_n$ above yields the result. \qed

\section{Proof of Theorem~\ref{thm:emp_process}}\label{app:thm3}

Let $\epsilon, \delta, \eta> 0$. By  \eqref{asy_linear} and since $\sup_{t \in \s{T}} |R_{n,t}| = \fasterthan(n^{-1/2})$, 
\[ n^{1/2}\left| \{\theta_n(t) - \theta_0(t)\} - \{\theta_n(s) - \theta_0(s)\}\right|\ \leq\ \left|\d{G}_n (\phi_{0,t} - \phi_{0,s}) \right| + \fasterthan(1)\ .\]
Condition (A2) implies that $\{\phi_{0,t} : t \in \s{T}\}$ is uniformly mean-square continuous, in the sense that 
\[\lim_{h \to 0} \sup_{\|t - s\| \leq h} \int \left\{\phi_{0,s}(x)-\phi_{0,t}(x)\right\}^2dP_0(x) = 0\ . \]
Since $\s{T}$ is totally bounded in $\| \cdot\|$, this also implies that $\{\phi_{0,t} : t \in \s{T}\}$ is totally bounded in the $L_2(P_0)$ metric. This, in addition to (A1), implies that $\{\d{G}_n \phi_{0,t} : t \in \s{T}\}$ converges weakly in $\ell^{\infty}(\s{T})$ to a Gaussian process $\d{G}$ with covariance function $\Sigma_0$. Furthermore, (A2) implies that this limit process is a tight element of $\ell^{\infty}(\s{T})$. By Theorem 1.5.4 of \cite{van1996weak}, $\{\d{G}_n \phi_{0,t} : t \in \s{T}\}$ is asymptotically tight. By Theorem 1.5.7 of \cite{van1996weak}, $\{\d{G}_n \phi_{0,t} : t \in \s{T}\}$ is thus asymptotically uniformly mean-square equicontinuous in probability, in the sense that there exists  some $\delta_0=\delta_0(\epsilon,\eta)>0$ such that \[\limsup_{n \to\infty} P_0\left\{ \sup_{\rho(s, t) <\delta_0} | \d{G}_n(\phi_{0,t} - \phi_{0,s}) | > \epsilon\right\}<\eta\]
with $\rho(s,t) := [\int \{\phi_{0,t}(x)-\phi_{0,s}(x)\}^2dP_0(x)]^{1/2}$. By (A2), $\sup_{\|t - s\| \leq h} \rho(t,s) < \delta_0$ for some $h>0$. Hence, for all $n$ large, both $\delta n^{-1/2} \leq h$ and  $ P_0\{ \sup_{\rho(s, t) <\delta_0} | \d{G}_n(\phi_{0,t} - \phi_{0,s}) | > \epsilon\} < \eta$, so that
\begin{align*}
P_0\left\{ \sup_{\|t - s\| \leq \delta n^{-1/2}} \left| \d{G}_n(\phi_{0,t} - \phi_{0,s})\right| > \epsilon\right\} \leq P_0\left\{ \sup_{\rho(t, s) < \delta_0} \left| \d{G}_n(\phi_{0,t} - \phi_{0,s})\right| > \epsilon\right\} < \eta\ ,
\end{align*}
which completes the proof. \qed

%

\vspace{.2in}

\singlespacing
\bibliographystyle{apa}
\bibliography{../projection_bib}

\clearpage

\section*{Supplementary Material}

Herein, we refer to \cite{van1996weak} as VW.  Throughout, the symbol $\lesssim$ should be interpreted to mean ``up to a constant not depending on $n$, $t$, $y$, or $a$".

We recall that the data unit is the vector $X = (A, Y)$, where $Y$ is an outcome and $A$ is now a continuous exposure. The observed data consist of independent draws $X_1, \dotsc, X_n$ from $P_0$. The parameter of interest $\theta_0$ is the conditional distribution function of $Y$ at $t_1$ given $A = t_2$, defined pointwise as $\theta_{0}(t_1, t_2) := P\left( Y \leq t_1 \mid A = t_2 \right)$. The local linear regression estimator $\theta_n$ is given by
\[ \theta_n(t_1, t_2) :=  \frac{1}{nh_n} \sum_{i=1}^n I(Y_j \leq t_1) \frac{s_{2,n}(t_2)  - s_{1,n}(t_2) \left(A_i - t_2\right)}{s_{0,n}(t_2)s_{2,n}(t_2) - s_{1,n}(t_2)^2} K\left( \frac{A_i - t_2}{h_n}\right) \ ,\]
where $K:\d{R}\to\d{R}$ is a symmetric and bounded kernel function, $h_n \to 0$ is a sequence of bandwidths, and $s_{j,n}(t_2) :=  \frac{1}{nh_n} \sum_{i=1}^n \left(A_i - t_2\right)^j K\left( \frac{A_i - t_2}{h_n}\right)$ for $j \in \{0,1,2\}$.  We also define
\begin{align*}
 R_{\theta}^{(2)}(t, \delta) &:= \theta_0(t_1, t_2 + \delta) - \theta_0(t_1, t_2) - \delta\theta_{0,t_2}'(t_1, t_2) - \tfrac{1}{2} \delta^2\theta_{0,t_2}''(t_1, t_2)\ , \\
 R_{f}^{(1)}(t, \delta) &:= f_0(t_2 +\delta ) - f_0(t_2) - \delta f_0'(t_2) \ .
 \end{align*}
 We recall the following conditions:
\begin{description}
\item[(d)] $\theta_{0,t_2}''$ exists and is continuous on $\s{T}$, and as $\delta \to 0$, $\sup_{t \in \s{T}} | R_{\theta}^{(2)}(t, \delta)| = \fasterthandet(\delta^2)$;
\item[(e)] $\inf_{t\in \s{T}} f_0(t_2) > 0$, $f_0'$ exists and is continuous on $\s{T}$, and $\sup_{t \in \s{T}}| R_{f}^{(1)}(t, \delta)|=  \fasterthandet(\delta)$;
\item[(f)] $K$ is a Lipschitz function supported on $[-1,1]$ and satisfies condition (M) of \cite{stupfler2016}.
\end{description} 
Letting 
\begin{align*}
\nu_{n,t}(y,a) &:=  \left[ I(y \leq t_1) - \theta_0(t_1, a)\right]  K\left( \frac{a - t_2}{h_n}\right) \text{ and } \\
R_n(t) &:= h_n^{-1/2} \left[ \frac{ s_{2,n}(t_2)}{g_n(t_2)} \d{G}_n \nu_{n,t} -  \frac{ s_{1,n}(t_2)}{g_n(t_2)} \d{G}_n \left( \ell_t \nu_{n,t}\right)\right]\ ,
\end{align*}
for $g_n(t_2) :=  s_{0,n}(t_2)s_{2,n}(t_2) - s_{1,n}(t_2)^2$,  the statement of Proposition~1 from the main text is:
\begin{proposition}\label{prop:kern_unif}
Suppose conditions (d)-(f) hold, $nh_n^5 =\boundeddet(1)$, and $nh_n^4/ \log h_n^{-1} \longrightarrow \infty$. Then
\[ \sup_{t \in \s{T}} \left| \left(nh_n \right)^{1/2} \left[ \theta_n(t_1, t_2) - \theta_0(t_1, t_2) \right] - \left(n h_n^5\right)^{1/2} \tfrac{1}{2}\theta_{0, t_2}''(t_1, t_2)  K_2 - R_n(t)  \right| \inprob 0 \ .\]
\end{proposition}

Recall further that condition (A) of the main text states:
\begin{description}
\item[(A)] there exists a deterministic sequence $r_n$ tending to infinity such that, for all $\delta>0$,
\[\sup_{\|t -s\| < \delta/r_n} \left|r_n\left[\theta_n(t) - \theta_0(t)\right] - r_n\left[\theta_n(s) - \theta_0(s)\right] \right| \inprob 0;\]
\end{description}
The statement of Proposition~2 from the main text is:
 \begin{proposition}\label{prop:cond_dist}
Suppose conditions (d)-(f) hold and $nh_n^5 \longrightarrow c \in (0, \infty)$. Then condition (A) holds for the local linear estimator with $r_n = (nh_n)^{1/2}$. 
\end{proposition}
We note that condition (M) of Stupfler (2016) guarantees that the class $\left\{ x \mapsto K\left( \frac{x - t}{h}\right) : h > 0, t \in \d{R}\right\}$ is VC with index 2.

We define $K_j := \int u^j K(u) \, du$ and
\begin{align*}
w_n(a, t_2) &:= s_{2,n}(t_2)  - s_{1,n}(t_2) (a - t_2) \\
w_0(a, t_2) &:= f_0(t_2) - f_0'(t_2) (a - t_2) \ .
\end{align*}
Before proving Propositions~1 and~2, we state and prove a Lemma that we will use.

\begin{lemma}\label{lemma:s12}
Suppose conditions (d)-(f) hold, $nh_n^4 \longrightarrow \infty$, and $nh_n^5 = \boundeddet(1)$. Then
\begin{align*}
\left(n h_n^5\right)^{1/2}\sup_{(t_1, t_2) \in \s{T}}  \left| \frac{s_{n,1}(t_2)}{g_n(t_2)} - \frac{f_0'(t_2)}{f_0(t_2)^2} \right| &\inprob 0 \ ,  \\
\left(n h_n^5\right)^{1/2}\sup_{(t_1, t_2) \in \s{T}}  \left| \frac{s_{n,2}(t_2)}{g_n(t_2)} - \frac{1}{f_0(t_2)} \right| &\inprob 0\ ,  \\
\left(n h_n^5\right)^{1/2}\sup_{(t_1, t_2) \in \s{T}} \sup_{|a - t_2| \leq h_n} \left| \frac{w_n(a, t_2)}{g_n(t_2)} - \frac{w_0(a, t_2)}{f_0(t_2)^2} \right| &\inprob 0\ ,
\end{align*}
and for any $\delta > 0$, 
\begin{align*}
\left( n h_n^5\right)^{1/2}\sup_{\|t - s\| \leq \delta / (n h_n)^{1/2}} \left|\frac{s_{1,n}(t_2)}{g_n(t_2)} - \frac{s_{1,n}(s_2)}{g_n(s_2)}\right| &\inprob 0\ \text{ and} \\
 \left( n h_n^4\right)\sup_{\|t - s\| \leq \delta / (n h_n)^{1/2}} \left|\frac{s_{2,n}(t_2)}{g_n(t_2)} - \frac{s_{2,n}(s_2)}{g_n(s_2)}\right| &\inprob 0\ .
\end{align*}
\end{lemma}
\begin{proof}[\bfseries{Proof of Lemma~\ref{lemma:s12}}]
We first show that $\sup_{t \in \s{T}} \left| s_{0,n}(t_2) -  f_0(t_2) \right| = \fasterthan(h_n)$. We have
\[ s_{0,n}(t_2)  - f_0(t_2) = h_n^{-1}\int K\left( \frac{a - t_2}{h_n}\right) f_0(a) \, da - f_0(t_2)+  n^{-1/2}h_n^{-1} \d{G}_n K\left( \frac{\cdot  - t_2}{h_n}\right) \ . \]
By the change of variables $u = (a - t_2) / h_n$, we have
\begin{align*}
h_n^{-1}\int K\left( \frac{a - t_2}{h_n}\right) f_0(a) \, da - f_0(t_2) &=  \int K\left(u\right) \left[ f_0(t_2 + h_n u) - f_0(t_2)\right] \, du \\
&=   h_n \int uK(u) (h_n u)^{-1} R_f^{(1)}\left((t_1,t_2), h_n u\right) \, du \ ,
\end{align*}
which tends to zero uniformly over $t_2$ faster than $h_n$ by the assumed uniform negligibility of $R_{f}^{(1)}$. For the second term, since $K$ is uniformly bounded and the class $\left\{ a \mapsto K\left( \frac{a - t_2}{h_n}\right) : t_2 \in [0,1]\right\}$ is $P_0$-Donsker, as implied by condition (M) of \cite{stupfler2016}, Theorem 2.14.1 of VW implies that $\sup_{t_2} \left| \d{G}_n K\left( \frac{\cdot  - t_2}{h_n}\right)\right| = \bounded(1)$. Then, since $n^{-1/2} h_n^{-1} = h_n \left( n h_n^4 \right)^{-1/2} = \fasterthan (h_n)$, this term is also $\fasterthan (h_n)$.

We next show that $\left(n h_n^5\right)^{1/2}\sup_{t \in \s{T}} \left|  h_n^{-2}s_{1,n}(t_2) -  f_0'(t_2) K_2 \right| = \fasterthan(1)$. We have
\begin{align*}
(n h_n)^{1/2} s_{1,n}(t_2) &= \left(n h_n^{-1}\right)^{1/2} \int (a - t_2) K\left( \frac{a - t_2}{h_n}\right) f_0(a) \, da \\
&\qquad\qquad+ h_n^{-1/2} \iint  (a - t_2) K\left( \frac{a - t_2}{h_n}\right) \, \d{G}_n(dy, da) \ .
\end{align*}
By the change of variables $u = (a - t_2) / h_n$, the first term equals
\begin{align*}
(nh_n^3)^{1/2} \int u K\left(u\right) f_0(t_2 + h_n u) \, du &= (nh_n^3)^{1/2} \int u K\left(u\right) \left[  f_0(t_2 + h_n u)  - f_0(t) - (h_n u) f_0'(t_2)\right] \, du \\
&\qquad\qquad +\left(nh_n^5\right)^{1/2}  f_0'(t_2) K_2 \\
&= (nh_n^5)^{1/2} \int u K(u) h_n^{-1}R_f^{(1)}\left((t_1,t_2), h_n u\right) \, du + (nh_n^5)^{1/2} f_0'(t_2) K_2 \ .
\end{align*}
By the assumed uniform negligibility of $R_{f}^{(1)}$ and since $h_n = O(n^{-1/5})$, we have that the first term tends to zero in probability uniformly over $t \in \s{T}$.

Turning to the second term in $s_{1,n}(t_2)$, we will apply Theorem 2.14.1 of VW to obtain a tail bound for the supremum of this empirical process over the one-dimensional class indexed by $t_2$. We note that, since $K$ is bounded by some $\bar{K}$ and supported on $[-1,1]$, 
\[ \left|(a - t_2) K\left( \frac{a - t_2}{h_n}\right) \right| \leq \bar{K} |a - t_2| I\left( |a - t_2| \leq h_n\right) \leq \bar{K} h_n \ . \]
Therefore, the class of functions 
\[\left\{ (y,a) \mapsto (a - t_2) K\left( \frac{a - t_2}{h_n}\right) : (t_1, t_2) \in \s{T} \right\} \]
has envelope $\bar{K} h_n$. Furthermore, since $(y,a) \mapsto (a - t_2)$ and $K$ are both uniformly bounded VC classes of functions, and $K$ is bounded, the class of functions possesses finite entropy integral. Hence, we have that
\[E_{0} \left[\sup_{(t_1, t_2) \in \s{T}}  \left|  h_n^{-1/2}\iint (a - t_2) K\left( \frac{a - t_2}{h_n}\right) \, \d{G}_n(dy, da) \right|\right]\ \leq\ C' h_n^{1/2}\longrightarrow 0\ .\] 
We now have that $(n h_n)^{1/2} \sup_{t \in \s{T}} \left| s_{1,n}(t_2) - h_n^2 f_0'(t_2) K_2 \right| = \fasterthan(1)$, which implies in particular that
\[ \sup_{t \in \s{T}} \left| s_{1,n}(t_2) \right| = (n h_n)^{-1/2} \fasterthan(1) + h_n^2 \bounded(1) = \bounded\left( \left[n h_n\right]^{-1/2}\right) \ . \]

Next, we show that $\left(n h_n^5\right)^{1/2}\sup_{t \in \s{T}} \left| h_n^{-2} s_{2,n}(t_2) - f_0(t_2) K_2 \right| = \fasterthan(h_n)$. The proof of this  is nearly identical to the preceding proof. We have
\begin{align*}
(n h_n)^{1/2} s_{2,n}(t_2) &= \left(n  h_n^{-1}\right)^{1/2} \int (a - t_2)^2 K\left( \frac{a - t_2}{h_n}\right) f_0(a) \, da \\
&\qquad\qquad+ h_n^{-1/2} \iint  (a - t_2)^2 K\left( \frac{a - t_2}{h_n}\right) \, \d{G}_n(dy, da) \ .
\end{align*}
By the change of variables $u = (a - t_2) / h_n$, the first term equals
\begin{align*}
(nh_n^5)^{1/2} \int u^2 K\left(u\right) f_0(t_2 + h_n u) \, du &= (nh_n^5)^{1/2} h_n \int u^3 K\left(u\right) \frac{R_f^{(1)}\left(t, h_n u\right)}{h_n u}\, du +\left(nh_n^5\right)^{1/2}  f_0(t_2) K_2 \ .
\end{align*}
The uniform negligibility of $R_{f}^{(1)}$ implies that the first term is $\fasterthan(h_n)$ uniformly in $t$.

Analysis of the second term in $s_{2,n}$ is analogous to that of $s_{1,n}$, except that the envelope function is now $\bar{K} h_n^{2}$, so that the empirical process term is  $\bounded\left(h_n^{3/2}\right)$. We also note that $\sup_{t_2} |s_{2,n}(t_2)| = \bounded\left( \left[nh_n\right]^{-1/2}\right)$. 

The above derivations imply that
\begin{align*}
 \left( nh_n^5\right)^{1/2}\sup_{t \in \s{T}} \left| h_n^{-2} g_n(t_2)- f_0(t_2)^2 K_2 \right| &\leq  \left( nh_n^5\right)^{1/2}\sup_{t \in \s{T}}\left| \left[ h_n^{-2} s_{2,n}(t_2) - f_0(t_2) K_2\right] s_{0,n}(t_2) \right|\\
 &\qquad + \left( nh_n^5\right)^{1/2}\sup_{t \in \s{T}}\left| \left[s_{0,n}(t_2) - f_0(t_2) \right] f_{0}(t_2) K_2\right| \\
 &\qquad + \left( nh_n\right)^{1/2} \left[ \sup_{t \in \s{T}} \left| s_{1,n}(t_2) \right|\right]^2 \\
 &= \fasterthan(1) \bounded(1) + \left( nh_n^5\right)^{1/2} \fasterthan(h_n) + \left( nh_n\right)^{1/2} \bounded\left( \left[n h_n\right]^{-1}\right) \\
 &= \fasterthan(1)\ .
\end{align*}
We now proceed to the statements in the Lemma. We write 
\begin{align*}
\left| \frac{s_{n,1}(t_2)}{g_n(t_2)} - \frac{f_0'(t_2)}{f_0(t_2)^2} \right| &= \left| \frac{h_n^{-2}s_{n,1}(t_2)}{h_n^{-2}g_n(t_2)} - \frac{f_0'(t_2)K_2}{f_0(t_2)^2 K_2} \right|  \\
&=\left| \frac{ h_n^{-2}s_{n,1}(t_2) - f_0'(t_2) K_2 }{h_n^{-2} g_n(t_2)} - f_0'(t_2) K_2 \frac{ h_n^{-2} g_n(t_2) -f_0(t_2)^2 K_2}{h_n^{-2} g_n(t_2) f_0(t_2)^2 K_2} \right| \\
&\leq  \frac{ \left|h_n^{-2}s_{n,1}(t_2) - f_0'(t_2) K_2\right|  }{h_n^{-2} g_n(t_2)} + f_0'(t_2) K_2 \frac{ \left| h_n^{-2} g_n(t_2) -f_0(t_2)^2 K_2 \right|}{h_n^{-2} g_n(t_2) f_0(t_2)^2 K_2} \ .
\end{align*}
Since $\inf_{t \in \s{T}} |f_0(t_2)| > 0$, $\sup_{t \in \s{T}} \left[h_n^{-2}g_n(t_2)\right]^{-1} = \bounded(1)$ and $\sup_{t\in\s{T}} \left[h_n^{-2} g_n(t_2) f_0(t_2)^2 \right]^{-1} = \bounded(1)$. The result follows.

We omit the proof of the statement regarding $s_{n,2}$, since it is almost identical to the above. For the statement regarding $w_n$, we have by the above calculations that
\[ \left( nh_n^5\right)^{1/2}\sup_{(t_1, t_2) \in \s{T}} \sup_{|a - t_2| \leq h_n} \left| h_n^{-2}w_n(a, t_2) -  w_0(a, t_2)  K_2 \right| \inprob 0 \ .\]
We write
\begin{align*}
\left| \frac{w_n(a, t_2)}{g_n(t_2)} - \frac{w_0(a, t_2)}{f_0(t_2)^2} \right| &= \left| \frac{h_n^{-2}w_n(a, t_2)}{h_n^{-2}g_n(t_2)} - \frac{w_0(a, t_2) K_2 }{f_0(t_2)^2 K_2} \right| \\
 & =\left| \frac{h_n^{-2}w_n(a, t_2)- w_0(a, t_2)K_2}{h_n^{-2}g_n(t_2)} - w_0(a, t_2) \frac{h_n^{-2}g_n(t_2) - f_0(t_2)^2 K_2}{h_n^{-2}g_n(t_2) f_0(t_2)^2} \right| \\
 &\leq \left[h_n^{-2}g_n(t_2)\right]^{-1} \left|h_n^{-2}w_n(a, t_2)- w_0(a, t_2)K_2\right| \\
 &\qquad+ \left|w_0(a, t_2)\right| \left[h_n^{-2} g_n(t_2) f_0(t_2)^2 \right]^{-1}  \left| h_n^{-2}g_n(t_2) - f_0(t_2)^2 K_2\right|  \ .
\end{align*}
The result follows.

We note that the above results imply that
\begin{align*}
\sup_{|t_2 - s_2| \leq \eta} |s_{1,n}(t_2)- s_{1,n}(s_2)| &\leq 2 \sup_{t_2} \left| s_{1,n}(t_2) - h_n^2 f_0'(t_2) K_2 \right| + h_n^2 \sup_{|t_2 - s_2| \leq \eta} | f_0'(t_2) - f_0'(s_2)| K_2\\
&\lesssim  \fasterthan\left( \left[n h_n\right]^{-1/2} \right) + h_n^2 \eta \ ,
\end{align*}
so that $\sup_{|t_2 - s_2| \leq \delta / (n h_n)^{-1/2}} |s_{1,n}(t_2)- s_{1,n}(s_2)| =    \fasterthan\left( \left[n h_n\right]^{-1/2} \right)$. Similarly, $\sup_{|t_2 - s_2| \leq \eta} |s_{2,n}(t_2)- s_{2,n}(s_2)|  =  \fasterthan\left( h_n \left[n h_n\right]^{-1/2} \right)$ and $\sup_{|t_2 - s_2| \leq \eta} |s_{0,n}(t_2)- s_{0,n}(s_2)|  =  \fasterthan\left( h_n  \right)$. Therefore,
\begin{align*}
\sup_{\|t - s\| \leq \delta / (n h_n)^{-1/2}} \left| g_n(t_2) - g_n(s_2)\right| &\leq \sup_{\|t - s\| \leq \delta / (n h_n)^{-1/2}} \left| \left[ s_{0,n}(t_2) - s_{0,n}(s_2) \right] s_{2,n}(s_2) \right|  \\
&\qquad + \sup_{\|t - s\| \leq \delta / (n h_n)^{-1/2}} \left| s_{0,n}(t_2) \left[ s_{2,n}(t_2) -s_{2,n}(s_2)\right] \right| \\
&\qquad + \sup_{\|t - s\| \leq \delta / (n h_n)^{-1/2}} \left| \left[ s_{1,n}(t_2)- s_{1,n}(s_2) \right]\left[ s_{1,n}(t_2)+ s_{1,n}(s_2) \right] \right| \\
&\lesssim  \fasterthan(h_n)  \bounded\left( \left[n h_n\right]^{-1/2}\right) + \bounded(1)\fasterthan\left( h_n \left[n h_n\right]^{-1/2} \right)\\
&\qquad +  \fasterthan\left( \left[n h_n\right]^{-1/2} \right)\bounded\left( \left[n h_n\right]^{-1/2}\right) \\
& = \fasterthan\left( h_n \left[n h_n\right]^{-1/2}\right) \ .
\end{align*}
We can now write
\begin{align*}
\sup_{\|t - s\| \leq \delta / (n h_n)^{-1/2}} \left|\frac{s_{1,n}(t_2)}{g_n(t_2)} - \frac{s_{1,n}(s_2)}{g_n(s_2)}\right| &\leq h_n^{-2} \sup_{\|t - s\| \leq \delta / (n h_n)^{-1/2}} \left| \frac{ s_{1,n}(t_2) - s_{1,n}(s_2)}{h_n^{-2}g_n(t_2)} \right| \\
&\qquad + h_n^{-4} \sup_{\|t - s\| \leq \delta / (n h_n)^{-1/2}} \left|s_{1,n}(s_2) \frac{ g_{n}(t_2) - g_{n}(s_2)}{h_n^{-2}g_n(t_2) h_n^{-2} g_n(s_2)} \right| \\
&= h_n^{-2} \fasterthan\left( \left[n h_n\right]^{-1/2} \right)  + h_n^{-4}  \bounded\left( \left[n h_n\right]^{-1/2}\right) \fasterthan\left(h_n \left[n h_n\right]^{-1/2}\right) \\
&= \fasterthan \left( \left[ n h_n^5\right]^{-1/2} \right) 
\end{align*}
and
\begin{align*}
\sup_{\|t - s\| \leq \delta / (n h_n)^{-1/2}} \left|\frac{s_{2,n}(t_2)}{g_n(t_2)} - \frac{s_{2,n}(s_2)}{g_n(s_2)}\right| &\leq h_n^{-2} \sup_{\|t - s\| \leq \delta / (n h_n)^{-1/2}} \left| \frac{ s_{2,n}(t_2) - s_{2,n}(s_2)}{h_n^{-2}g_n(t_2)} \right| \\
&\qquad + h_n^{-4} \sup_{\|t - s\| \leq \delta / (n h_n)^{-1/2}} \left|s_{2,n}(s_2) \frac{ g_{n}(t_2) - g_{n}(s_2)}{h_n^{-2}g_n(t_2) h_n^{-2} g_n(s_2)} \right| \\
&= h_n^{-2} \fasterthan\left( h_n \left[n h_n\right]^{-1/2}\right)  \\
&\qquad+ h_n^{-4}\bounded\left( \left[n h_n\right]^{-1/2}\right) \fasterthan\left(h_n \left[n h_n\right]^{-1/2}\right) \\
&= \fasterthan \left( \left[ n h_n^4\right]^{-1} \right) \ .
\end{align*}

\end{proof}

\begin{proof}[\bfseries{Proof of Proposition~1}]
We define
\begin{align*}
m_{1,n}(t_1, t_2) &:=  h_n^{-1}\iint \left[ \theta_0(t_1, a) - \theta_0(t_1, t_2) \right]\frac{w_n(a, t_2)}{g_n(t_2)} K\left( \frac{a - t_2}{h_n}\right) \, \d{P}_n(dy, da) \\
m_{2,n}(t_1, t_2) &:= h_n^{-1} \iint \left[ I(y \leq t_1) -  \theta_0(t_1, a) \right]\frac{w_n(a, t_2)}{g_n(t_2)} K\left( \frac{a - t_2}{h_n}\right) \, \d{P}_n(dy, da) \ .
\end{align*}
Then $\theta_n(t_1, t_2) - \theta_0(t_1, t_2) = m_{1,n}(t_1, t_2) + m_{2,n}(t_1, t_2)$. We note that since $E_0\left[ I(Y \leq t_1) \mid A = a\right] =   \theta_0(t_1, a)$, we have 
\begin{align*}
(nh_n)^{1/2} m_{2,n}(t_1, t_2) &= h_n^{-1/2}\iint \left[ I(y \leq t_1) -  \theta_0(t_1, a) \right]\frac{w_n(a, t_2)}{g_n(t_2)} K\left( \frac{a - t_2}{h_n}\right) \, \d{G}_n(dy, da) \\
&= h_n^{-1/2} \left[ \frac{s_{2,n}(t_2)}{g_n(t_2)} \d{G}_n v_{n,t} - \frac{s_{1,n}(t_2)}{g_n(t_2)} \d{G}_n \left( \ell_t v_{n,t} \right)\right]
\end{align*}
Therefore,
\begin{align*}
& \left(nh_n \right)^{1/2} \left[ \theta_n(t_1, t_2) - \theta_0(t_1, t_2) \right] - \left(n h_n^5\right)^{1/2} \tfrac{1}{2}\theta_{0, t_2}''(t_1, t_2)  K_2 - R_n(t_1, t_2) \\
&\qquad = \left(nh_n \right)^{1/2} m_{1,n}(t_1, t_2) - \left(n h_n^5\right)^{1/2} \tfrac{1}{2}\theta_{0, t_2}''(t_1, t_2)  K_2 \ .
\end{align*}
 We now proceed to analyze $m_{1,n}$. We have that
\begin{align*}
\left(nh_n\right)^{1/2} m_{1,n}(t_1, t_2) &= \left(n h_n^{-1}\right)^{1/2} \int \left[ \theta_0(t_1, a) - \theta_0(t_1, t_2) \right] \frac{w_n(a, t_2)}{g_n(t_2)} K\left( \frac{a - t_2}{h_n}\right)  f_0(a)\,  da \\ 
&\qquad +  h_n^{-1/2}\iint \left[ \theta_0(t_1, a) - \theta_0(t_1, t_2) \right] \frac{w_n(a, t_2)}{g_n(t_2)} K\left( \frac{a - t_2}{h_n}\right) \, \d{G}_n(dy, da) \ .
\end{align*}
The second term in $m_{1,n}$ may be further decomposed as 
\begin{align*}
 h_n^{-1/2}\frac{s_{2,n}(t_2)}{g_n(t_2)} \d{G}_n \gamma_{t,n} - h_n^{-1/2}\frac{s_{1,n}(t_2)}{g_n(t_2)} \d{G}_n \left( \ell_t \gamma_{t,n}\right)
\end{align*}
for $\gamma_{t,n}(y,a) := \left[ \theta_0(t_1, a) - \theta_0(t_1, t_2) \right] K\left( \frac{a - t_2}{h_n}\right)$ and $\ell_t(y,a) := a - t_2$. By Lemma~\ref{lemma:s12}, $\sup_{t\in \s{T}} \left|\frac{s_{2,n}(t_2)}{g_n(t_2)}\right| = \bounded\left( \left[n h_n^5\right]^{-1/2}\right)$, and similarly for $s_{1,n}$. We will use Theorem 2.14.2 of VW to obtain bounds for $\sup_{t \in \s{T}} \left| \d{G}_n \gamma_{t,n}\right|$ and $\sup_{t \in \s{T}} \left| \d{G}_n \left( \ell_t \gamma_{t,n}\right)\right|$. We first note that, since $K$ is bounded and supported on $[-1,1]$ and $\theta_0$ is Lipschitz on $\s{T}$, $\sup_{t \in \s{T}} | \gamma_{t,n}| \lesssim h_n$ and $\sup_{t \in \s{T}} | \ell_t \gamma_{t,n}| \lesssim h_n^2$. These will be our envelope functions for these classes. Next, since $K$ is Lipschitz, we have that 
\begin{align*}
\left| \gamma_{t,n} - \gamma_{s,n} \right| &\leq \left| \left[ \theta_0(t_1, a) - \theta_0(s_1, a)  \right ] - \left[\theta_0(t_1, t_2)  - \theta_0(s_1, s_2) \right] \right|   K\left( \frac{a - s_2}{h_n}\right) \\
&\qquad +\left| \theta_0(t_1, a) - \theta_0(t_1, t_2) \right| \left| K\left( \frac{a - t_2}{h_n}\right) - K\left( \frac{a - s_2}{h_n}\right) \right| \\
&\lesssim |t_1 - s_1| + \|t - s\| + |t_2 - s_2| h_n^{-1} \lesssim \|t-s\| h_n^{-1} \ .
\end{align*}
Therefore, by VW Theorem 2.7.11, we have $N_{[]}\left( 2\varepsilon h_n^{-1}, \s{G}_n, L_2(P_0)\right) \lesssim N(\varepsilon, \s{T}, \|\cdot\|) \lesssim \varepsilon^{-2}$, where $\s{G}_n := \{ \gamma_{n,t} : t\in \s{T}\}$. Thus, by VW Theorem 2.14.2, 
\begin{align*}
\sup_{t\in\s{T}} | \d{G}_n \gamma_{n,t} | &\lesssim \int_0^1 \left[ N_{[]}\left(\varepsilon h_n, \s{G}_n, L_2(P_0) \right)\right]^{1/2} \, d\varepsilon  \,h_n \lesssim  \int_0^1 \left[-\log (\varepsilon h_n^2)\right]^{1/2} \, d\varepsilon  \,h_n \\
&=  h_n^{-1} \int_0^{h_n^2} \left[ -\log \varepsilon \right]^{1/2} \, d\varepsilon  \lesssim h_n^{-1} \left\{ h_n^2 \left[ \log \left(h_n^{-2} \right) \right]^{1/2} \right\}  \lesssim  h_n \left(\log h_n^{-1} \right)\ ,
\end{align*}
where we have used the fact that $\int_0^z \left[ \log x^{-1}\right]^{1/2} \, dx \lesssim z \left[ \log z^{-1} \right]^{1/2}$ for all $t$ small enough. A similar argument applies to $\sup_{t\in\s{T}} | \d{G}_n (\ell_t \gamma_{n,t})|$. We thus have that the second term in $m_{1,n}$ is bounded above up to a constant not depending on $n$ and uniformly in $t$  by 
\[ h_n^{-1/2}  \bounded\left( \left[n h_n^5\right]^{-1/2}\right)  h_n \left(\log h_n^{-1}\right)^{1/2} =   \bounded\left( \left[\frac{n h_n^4}{ \log h_n^{-1}} \right]^{-1/2}\right) \ ,\] 
which is $\fasterthan(1)$ since $\frac{n h_n^4}{ \log h_n^{-1}} \to \infty$.
%

By the change of variables $u = (a -t_2) / h_n$, the first term in $m_{1,n}$ equals
\begin{align*}
&\left(nh_n\right)^{1/2}\int \left[ \theta_0(t_1, t_2+ h_n u) - \theta_0(t_1, t_2) \right] \frac{w_n(t_2 + h_n u, t_2)}{g_n(t_2)} K\left(u\right)  f_0(t_2 + h_n u)\,  du \\
&= (nh_n)^{1/2} \int \left[ R_{\theta}^{(2)}(t, h_n u)  +\theta_{0,t_2}'(t_1, t_2) (h_n u) +  \tfrac{1}{2} \theta_{0,t_2}''(t_1, t_2) (h_n u)^2 \right]  \\
&\qquad \qquad\qquad\qquad \cdot \left[R_{f}^{(1)}(t_2, h_n u) + f_0(t_2) + f_0'(t_2) (h_n u) \right] \frac{w_n(t_2 + h_n u, t_2)}{g_n(t_2)}K(u)\, du \ .
\end{align*}
Expanding the product, this is equal to 
\begin{align*}
&(nh_n)^{1/2} \int (h_nu) \left[\theta_{0,t_2}'(t_1, t_2) + \tfrac{1}{2} \theta_{0,t_2}''(t_1, t_2) (h_n u)\right]\left[ f_0(t_2) + f_0'(t_2) h_n u\right] \frac{w_n(t_2 + h_n u, t_2)}{g_n(t_2)}K(u)\, du \\
&\quad + (nh_n)^{1/2} \int  R_{\theta}^{(2)}(t, h_n u) \left[ f_0(t_2) + f_0'(t_2) h_n u\right] \frac{w_n(t_2 + h_n u, t_2)}{g_n(t_2)}K(u)\,du\\
&\quad + (nh_n)^{1/2} \int R_{f}^{(1)}(t_2, h_n u) (h_n u)\left[ \theta_{0,t_2}'(t_1, t_2) + \tfrac{1}{2} \theta_{0,t_2}''(t_1, t_2) h_n u\right]  \frac{w_n(t_2 + h_n u, t_2)}{g_n(t_2)}K(u)\,du\\
&\quad+ (nh_n)^{1/2} \int R_{\theta}^{(2)}(t, h_n u)R_{f}^{(1)}(t_2, h_n u)  \frac{w_n(t_2 + h_n u, t_2)}{g_n(t_2)}K(u)\,du\ .
\end{align*}
By the assumed negligibility of $R_{\theta}^{(2)}$ and $R_f^{(1)}$ and Lemma~\ref{lemma:s12}, the second through fourth terms tend to zero in probability uniformly over $\s{T}$. The first term equals
\begin{align*}
&\int f_0'(t_2) \left[\theta_{0,t_2}'(t_1, t_2) + \tfrac{1}{2} \theta_{0,t_2}''(t_1, t_2) (h_n u)\right] \left( n h_n^5\right)^{1/2}\left[ \frac{w_n(t_2 + h_n u, t_2)}{g_n(t_2)} -  \frac{w_0(t_2 + h_n u, t_2)}{f_0(t_2)^2}\right] u^2 K(u) \, du \\
&\quad + \left( n h_n^5\right)^{1/2} \int f_0'(t_2)  \left[\theta_{0,t_2}'(t_1, t_2) + \tfrac{1}{2} \theta_{0,t_2}''(t_1, t_2) (h_n u)\right] \frac{w_0(t_2 + h_n u, t_2)}{f_0(t_2)^2}  u^2 K(u) \, du \\
&\quad + \left( n h_n^3\right)^{1/2} \int f_0'(t_2)  \left[\theta_{0,t_2}'(t_1, t_2) + \tfrac{1}{2} \theta_{0,t_2}''(t_1, t_2) (h_n u)\right] \frac{w_n(t_2 + h_n u, t_2)}{g_n(t_2)}  u K(u) \, du \ .
\end{align*}
By Lemma~\ref{lemma:s12}, the first term tends to zero uniformly over $\s{T}$. By symmetry of $K$, the second plus third terms simplifies to
\begin{align*}
&\left( n h_n^5\right)^{1/2}\tfrac{1}{2}\theta_{0,t_2}''(t_1, t_2) K_2  + \left( n h_n^5\right)^{1/2}  \left[ \frac{s_{2,n}(t_2)}{g_n(t_2)} - \frac{1}{f_0(t_2)} \right]f_0(t_2) \tfrac{1}{2}\theta_{0,t_2}''(t_1, t_2) K_2 \\
&\qquad-\left( n h_n^5\right)^{1/2} \left[ \frac{s_{1,n}(t_2)}{g_n(t_2)} - \frac{f_0'(t_2)}{f_0(t_2)^2} \right] \theta_{0,t_2}'(t_1, t_2) f_0(t_2) K_2 \ .
\end{align*}
Once again, the second and third summands tend to zero uniformly over $\s{T}$ by Lemma~\ref{lemma:s12}. We have now shown that
\[ \sup_{t \in \s{T}} \left| \left(nh_n \right)^{1/2} m_{1,n}(t_1, t_2) - \left(n h_n^5\right)^{1/2} \tfrac{1}{2}\theta_{0, t_2}''(t_1, t_2)  K_2 \right| \inprob 0 \ ,\]
which completes the proof.
\end{proof}

We can now prove Proposition~2.
\begin{proof}[\bfseries{Proof of Proposition~2}]
Since $\theta_{0, t_2}''$ is uniformly continuous and $n h_n^5 = O(1)$, 
\[\sup_{\| t - s\| \leq \delta / r_n} \left| \left(n h_n^5\right)^{1/2} \tfrac{1}{2}\theta_{0, t_2}''(t_1, t_2)K_2 -  \left(n h_n^5\right)^{1/2} \tfrac{1}{2}\theta_{0, t_2}''(s_1, s_2) K_2 \right| \inprob 0 \ .\] 
Therefore, it only remains to show that $\sup_{\|t - s\| \leq \delta / r_n} \left| R_n(t) - R_n(s) \right| \inprob 0.$
Recalling that $\ell_t(y,a) := a - t_2$ and $\nu_{n,t}(y,a) :=  \left[ I(y \leq t_1) - \theta_0(t_1, a)\right] K\left( \frac{a - t_2}{h_n}\right)$, we have
\begin{align*}
 R_n(t) -R_n(s) &=\left[  \frac{s_{2,n}(t_2)}{g_n(t_2)} - \frac{s_{2,n}(s_2)}{g_n(s_2)} \right]\d{G}_n \nu_{n,t} -\left[  \frac{s_{1,n}(t_2)}{g_n(t_2)} - \frac{s_{1,n}(s_2)}{g_n(s_2)} \right]\d{G}_n\left( \ell_t \nu_{n,t} \right) \\
 &\qquad\qquad+ \frac{s_{2,n}(s_2)}{g_n(s_2)} \d{G}_n \left( \nu_{n,t} - \nu_{n,s}\right) - \frac{s_{1,n}(s_2)}{g_n(s_2)}  \d{G}_n \left(\ell_{t} \nu_{n,t} - \ell_s\nu_{n,s}\right) \ .
\end{align*}
Focusing first on $\d{G}_n \nu_{n,t}$, we have $\d{G}_n \nu_{n,t}= \d{G}_n \nu_{n,t,1}-  \d{G}_n \nu_{n,t,2}$ for $\nu_{n,t,1}(y,a) = I(y \leq t_1) K\left( \frac{a - t_2}{h_n}\right)$ and $\nu_{n,t,2}(y,a) = \theta_0(t_1, a) K\left( \frac{a - t_2}{h_n}\right)$. The classes $\left\{ I(y \leq t_1) : t \in \s{T} \right\}$ and $\left\{ K\left( \frac{a - t_2}{h_n}\right) : t \in \s{T} \right\}$ are both uniformly bounded above and VC. Therefore, the uniform covering numbers of the class $\left\{ I(y \leq t_1)K\left( \frac{a - t_2}{h_n}\right)  : t \in \s{T} \right\}$ are bounded up to a constant by $\varepsilon^{-V}$ for some $V < \infty$, so that the uniform entropy integral satisfies $J(\eta, \s{G}_{n,1}) \lesssim \eta \left( \log \eta^{-1} \right)^{1/2}$ for all $\eta$ small enough, where $\s{G}_{n, 1} := \{ \nu_{n,t,1} : t \in\s{T}\}$. We also have $P_0  \left(\nu_{n,t,1}\right)^2 \lesssim h_n$ for all $t \in \s{T}$ and all $n$ large enough. Thus, Theorem 2.1 of \cite{vandervaart2011} implies that 
\[ \sup_{t \in \s{T}} \left| \d{G}_n \nu_{n,t,1} \right| \lesssim h_n^{1/2}\left( \log h_n^{-1} \right)^{1/2} + n^{-1/2} \log h_n^{-1} \ .\]
For $\d{G}_n \nu_{n,t,2}$, we have that 
\[\left|\nu_{n,t,2}(y, a) - \nu_{n,s,2}(y,a) \right| \lesssim\|t-s\| (1 + h_n^{-1}) \lesssim h_n^{-1}\|t-s\| \]
for all $n$ large enough and all $(y,a)$. We can therefore apply Theorem 2.7.11 of VW to conclude that $N_{[]}\left(2\varepsilon h_n^{-1}, \s{G}_{n,2}, L_2(P_0) \right) \lesssim \varepsilon^{-2}$ for all $\varepsilon$ small enough, where $\s{G}_{n,2} = \left\{ \nu_{n,t,2} : t \in \s{T} \right\}$, which implies that  $N_{[]}\left(\varepsilon, \s{G}_{n,2}, L_2(P_0) \right) \lesssim (\varepsilon h_n)^{-2}$. Thus, $J_{[]}(\eta, \s{G}_{n,2}) \lesssim \eta \left[ \log (\eta h_n)^{-1} \right]^{1/2}$. Since $P_0  \left(\nu_{n,t,2}\right)^2 \lesssim  h_n$ as well, by Lemma 3.4.2 of VW, we then have
\[ E_{P_0} \sup_{t \in \s{T}} \left| \d{G}_n \nu_{n,t,2} \right| \lesssim h_n^{1/2}\left( \log h_n^{-1} \right)^{1/2} + n^{-1/2} \log h_n^{-1} \ .\]
Combining these two bounds with the last statement of Lemma~\ref{lemma:s12} yields
\begin{align*}
h_n^{-1/2}\sup_{t\in\s{T}}  \left| \left[  \frac{s_{2,n}(t_2)}{g_n(t_2)} - \frac{s_{2,n}(s_2)}{g_n(s_2)} \right]\d{G}_n \nu_{n,t} \right| &\lesssim h_n^{-1/2}\fasterthan \left( \left[ n h_n^4\right]^{-1} \right) \bounded\left( h_n^{1/2}\left[ \log h_n^{-1} \right]^{1/2} + n^{-1/2} \log h_n^{-1}\right) \\
&= \fasterthan(1) \left[ \frac{n h_n^4}{\left( \log h_n^{-1}\right)^{1/2}}\right]^{-1} + \fasterthan(1) \left( nh_n^{11/3}\right)^{-3/2} h_n \log h_n^{-1} \ .
\end{align*} 
Both terms tend to zero.

The analysis for $\d{G}_n\left( \ell_t \nu_{n,t}\right)$ is very similar. In this case, we have $P_0  \left(\ell_t \nu_{n,t}\right)^2 \lesssim h_n^3$, so that, using the same approach as above, we get
\[ E_{P_0} \sup_{t \in \s{T}} \left| \d{G}_n\left( \ell_t \nu_{n,t} \right) \right| \lesssim h_n^{3/2} \left( \log h_n^{-1} \right)^{1/2} + n^{-1/2} \left( \log h_n^{-1} \right)^{1/2}\]
and therefore, in view of Lemma~\ref{lemma:s12},
\begin{align*}
&h_n^{-1/2}\sup_{t\in\s{T}}  \left| \left[  \frac{s_{1,n}(t_2)}{g_n(t_2)} - \frac{s_{1,n}(s_2)}{g_n(s_2)} \right]\d{G}_n \left( \ell_t \nu_{n,t} \right) \right| \\
&\qquad\qquad\lesssim h_n^{-1/2} \fasterthan \left( \left[ n h_n^5\right]^{-1/2} \right) \bounded\left( h_n^{3/2} \left[ \log h_n^{-1} \right]^{1/2} + n^{-1/2} \left( \log h_n^{-1} \right)^{1/2}\right) \\
&= \fasterthan(1) \left( nh_n^4\right)^{-1/2} \left(h_n \log h_n^{-1} \right)^{1/2}+ \fasterthan(1) \left( nh_n^3\right)^{-1/2} \left(h_n \log h_n^{-1} \right)^{1/2} \ ,
\end{align*}
which goes to zero in probability.

It remains to bound $\sup_{\|t - s\| < \delta / r_n} \left|  \d{G}_n \left( \nu_{n,t} - \nu_{n,s}\right)\right|$ and $\sup_{\|t - s\| < \delta / r_n} \left|  \d{G}_n \left( \ell_t \nu_{n,t} - \ell_s \nu_{n,s}\right)\right|$. For the former, we work on the terms $\d{G}_n \left(\nu_{n,t,1} -  \nu_{n,s,1} \right)$ and $\d{G}_n \left(\nu_{n,t,2} -  \nu_{n,s,2}\right)$ separately. For the first of these,  we let $\s{F}_{n,\delta,2} := \left\{ \nu_{n,t,1}  - \nu_{n,s,1} : \| t- s\| \leq \delta / r_n\right\}$. We have
\begin{align*}
\left\|\nu_{n,t,1}  - \nu_{n,s,1} \right\|_{P_0, 2} & \leq \left( E_{P_0} \left\{ \left[  I(Y \leq t_1) - I(Y \leq s_1) \right]^2 K\left( \frac{A - s_2}{h_n}\right)^2 \right\} \right)^{1/2} \\
&\qquad +  \left( E_{P_0} \left\{ I(Y \leq t_1) \left[K\left( \frac{A - t_2}{h_n}\right) - K\left( \frac{A - s_2}{h_n}\right)\right]^2 \right\} \right)^{1/2} \\
&\lesssim \left( E_{P_0} \left\{  I(s_1 < Y \leq t_1)   I\left(|A - s_2| \leq h_n\right) \right\} \right)^{1/2} +h_n^{-1} |t_2 - s_2| \\
&\lesssim h_n^{1/2} \left| t_1 - s_1\right|^{1/2}  + h_n^{-1}\left|t_2 - s_2\right| \ .
\end{align*}
Therefore, $\sup_{f \in \s{F}_{n,\delta,1}} \left( P_0 f^2 \right)^{1/2}  \lesssim \left(nh_n^{-1}\right)^{-1/4} + \left(n h_n^3\right)^{-1/2} \lesssim \left(n h_n^3\right)^{-1/2}$ for all $n$ large enough. In addition, $\s{F}_{n,\delta,1}$ has uniform covering numbers bounded up to a constant by $\varepsilon^{-V}$ for all $n$ and $\delta$ because the classes $\left\{ I(y \leq t_1) : t \in \s{T} \right\}$ and $\left\{ K\left( \frac{a - t_2}{h_n}\right): t \in \s{T} \right\}$ are VC. Therefore, $J\left( \eta, \s{F}_{n,\delta,1} \right) \lesssim \eta \left( \log \eta^{-1} \right)^{1/2}$ for all $\eta$ small enough. Thus, Theorem 2.1 of \cite{vandervaart2011} implies that 
\[ E_{P_0} \sup_{\| t- s\| \leq \delta / r_n} \left| \d{G}_n \left( \nu_{n,t,1}  - \nu_{n,s,1} \right)\right| \lesssim \left(n h_n^3\right)^{-1/2} \left( \log \left[  n h_n^3\right] \right)^{1/2} + n^{-1/2}  \log \left[  n h_n^3\right] \ . \]
Turning to $\d{G}_n \left(\nu_{n,t,2} -  \nu_{n,s,2}\right)$, we analogously define $\s{F}_{n,\delta,2} := \left\{ \nu_{n,t,2} - \nu_{n,s,2} : \| t- s\| \leq \delta / r_n\right\}$. We have by the Lipschitz nature of $\theta_0$ and $K$ that
\[ \left| \theta_0(t_1, a) K\left( \frac{a - t_2}{h_n}\right) - \theta_0(s_1, a) K\left( \frac{a - s_2}{h_n}\right) \right|  \lesssim h_n^{-1} \| t- s\| \ . \]
Therefore, an envelope function $F_{n,\delta,2}$ for $\s{F}_{n,\delta,2}$ is given (up to a constant) by $h_n^{-1} \delta / r_n\lesssim \left(n h_n^3 \right)^{-1/2}$. Next, we have for any $(t, s)$ and $(t', s')$ in $\s{T}^2$ 
\begin{align*}
&\left| \left[ \theta_0(t_1, a) K\left( \frac{a - t_2}{h_n}\right) - \theta_0(s_1, a) K\left( \frac{a - s_2}{h_n}\right)  \right] \right. \\
&\left.\qquad - \left[ \theta_0(t_1', a) K\left( \frac{a - t_2'}{h_n}\right) - \theta_0(s_1', a) K\left( \frac{a - s_2'}{h_n}\right)  \right] \right| \\
&\qquad\qquad \leq \left| \theta_0(t_1, a)  - \theta_0(t_1', a) \right| K\left( \frac{a - t_2'}{h_n}\right) + \left| \theta_0(t_1, a)\right| \left| K\left( \frac{a - t_2}{h_n}\right) - K\left( \frac{a - t_2'}{h_n}\right)\right| \\
&\qquad\qquad\qquad  +  \left| \theta_0(s_1, a)  - \theta_0(s_1', a) \right| K\left( \frac{a - s_2'}{h_n}\right) + \left| \theta_0(s_1, a)\right| \left| K\left( \frac{a - s_2}{h_n}\right) - K\left( \frac{a - s_2'}{h_n}\right)\right|  \\
&\qquad\qquad\lesssim |t_1 - t_1'| + h_n^{-1} |t_2 - t_2'| + |s_1 - s_1'| + h_n^{-1} |s_2 - s_2'| \\
&\qquad\qquad\lesssim h_n^{-1} \left\|(t,s) - (t', s') \right\|_{\s{T}^2} \ ,
\end{align*}
where $ \left\|(t,s) - (t', s') \right\|_{\s{T}^2} := \max\{ \| t - s\|, \|t' - s'\| \}$. We therefore have by VW Theorem 2.7.11 that $N_{[]} \left( 2 \varepsilon h_n^{-1}, \s{F}_{n,\delta,2}, L_2(P_0) \right) \leq N\left( \varepsilon, \s{U}_{\delta / r_n}, \left\| \cdot \right\|_{\s{T}^2} \right)$, where $\s{U}_{\delta / r_n} := \left\{(t,s) \in \s{T}^2 : \| t- s\| \leq \delta / r_n \right\}$. Since $\s{U}_{\delta / r_n} \subseteq \s{T}^2$, we trivially have $N\left( \varepsilon, \s{U}_{\delta / r_n}, \left\| \cdot \right\|_{\s{T}^2} \right) \lesssim \varepsilon^{-4}$. Thus, 
\[ N_{[]} \left( \varepsilon \left[ nh_n^3 \right]^{-1/2}, \s{F}_{n,\delta,2}, L_2(P_0) \right) \lesssim \left( \varepsilon \left[ nh_n \right]^{-1/2}\right)^{-4} \ . \]
Therefore, VW Theorem 2.14.2 implies that 
\begin{align*}
 E_{P_0} \sup_{\| t - s\| \leq \delta / r_n} \left| \d{G}_n \left( \nu_{n,t,2}  - \nu_{n,s,2}\right) \right| &\lesssim \left(n h_n^3 \right)^{-1/2} \int_0^1 \left[ \log \left( \varepsilon [nh_n]^{-1/2}\right)^{-1} \right]^{1/2} \, d\varepsilon  \\
 &= \left(n h_n^3 \right)^{-1/2} \left(n h_n \right)^{1/2} \int_0^{(nh_n)^{-1/2}}  \left[ \log u^{-1} \right]^{1/2} \, du \\
 &\lesssim \left(n h_n^3 \right)^{-1/2} \left[ \log (n h_n) \right]^{1/2} \ .
 \end{align*}
 We now have that
 \begin{align*}
&h_n^{-1/2} \sup_{\| t -s\| \leq \delta / r_n}\left|  \frac{s_{2,n}(s_2)}{g_n(s_2)} \d{G}_n \left( \nu_{n,t} - \nu_{n,s}\right) \right| \\
&\qquad \lesssim h_n^{-1/2} \bounded\left(\left[nh_n^5\right]^{-1/2} \right) \bounded\left( \left(n h_n^3\right)^{-1/2} \left( \log \left[  n h_n^3\right] \right)^{1/2} + n^{-1/2}  \log \left[  n h_n^3\right] \right) \\
&= \bounded(1) \left[ \left( nh_n^{9/2} \right)^{-1} \left( \log \left[  n h_n^3\right] \right)^{1/2}+\left(n h_n^3\right)^{-1}  \log \left[  n h_n^3\right] \right] \ .
 \end{align*}
Both terms tend to zero in probability.

Finally, we address $\sup_{\|t - s\| < \delta / r_n} \left|  \d{G}_n \left( \ell_t \nu_{n,t} - \ell_s \nu_{n,s}\right)\right|$ in a very similar manner. As before, we work on the terms $\d{G}_n  \left(\ell_t \nu_{n,t,1}-  \ell_s \nu_{n,s,1} \right)$ and $\d{G}_n \left(\ell_t \nu_{n,t,2}-  \ell_s \nu_{n,s,2}\right)$ separately. It is straightforward to see that the same line of reasoning as used above applies to each of these terms as well, yielding the same negligibility.

\end{proof}

\end{document}